\documentclass[11pt,letterpaper]{amsart}

\usepackage[utf8]{inputenc}
\usepackage{amssymb, amscd, amsmath, epsfig, mathtools}
\usepackage{amsthm}
\usepackage{enumerate}
\usepackage{xcolor}
\usepackage{scalerel}
\usepackage{soul}
\usepackage{tikz-cd}
\usepackage{esint}
\usepackage{cancel}
\usepackage{hyperref}
\usepackage{slashed}
\usepackage{url}

\usepackage[margin=1.0in]{geometry}

\newtheorem{theorem}{Theorem}[section]
\newtheorem*{theorem*}{Theorem}

\newtheorem{corollary}{Corollary}[section]
\newtheorem*{corollary*}{Corollary}
\newtheorem{lemma}{Lemma}[section]
\newtheorem{proposition}{Proposition}[section]
\theoremstyle{definition}
\newtheorem{remark}{Remark}[section]

\newcommand{\R}{\mathbb R}

\newcommand{\calC}{\mathcal C}

\newcommand{\calL}{\mathcal L}
\newcommand{\dvol}{ \text{dvol}_{g}}
\newcommand{\dvoll}{d\text{vol}_{\hat{g}}}

\begin{document}

\title[Positivity on Hypersurfaces of Noncompact Cylinders]{Existence of Positive Scalar Curvature and Positive Yamabe constant on Hypersurfaces of Noncompact Cylinders}
\author{Jie Xu}
\address{Department of Mathematics, Northeastern University, Boston, MA, USA}
\email{jie.xu@northeastern.edu}

\begin{abstract}
Let $ X $ be an oriented, closed manifold with $ \dim X \geqslant 2 $. Let $ (Z, \partial Z) $ be an oriented, compact manifold with (possibly empty) smooth boundary and $ \dim Z \geqslant 2 $. In this article, we show that if the noncompact cylinder $ X \times \R $ admits a complete Riemannian metric $ g $ with positive injectivity radius and uniformly positive scalar curvature, and that is of bounded geometry or bounded curvature, then $ X $ admits a positive scalar curvature metric within the same conformal class provided that some $ g $-angle condition is satisfied. This partially answers a conjecture of Rosenberg and Stolz \cite{RosSto} without topological assumptions. With the $ g $-angle condition, we can also show that if $ (Z \times \R, \partial Z \times \R) $ admits a complete metric $ g $ that has positive Yamabe constant and positive injectivity radius, and is of bounded geometry or bounded curvature, then $ (Z, \partial Z) $ has positive Yamabe constant for the conformal class $ [\imath^{*}g] $ with the natural inclusion $ \imath : (Z, \partial Z) \hookrightarrow (Z \times \R, \partial Z \times \R) $.
\end{abstract}

\maketitle

\section{Introduction}
Let $ X $ be a closed, oriented manifold with $ \dim X \geqslant 2 $. Let $ M = X \times \R $ be the noncompact cylinder. We assign the $ \xi $-variable to the real line $ \R $ as global coordinate. In a 1994 survey article \cite{RosSto}, Rosenberg and Stolz conjectured the (non)existence of positive scalar curvature (PSC) metrics on the compact cylinder $ X $:
\medskip

{\it{Conjecture: If $ X $ does not admit any PSC metric, then $ M = X \times \R $ does not admit a complete PSC metric.}}
\medskip

In this article, we first focus on the following question, which is the contrapositive statement of this Rosenberg-Stolz conjecture:
\medskip

{\it{If $ M = X \times \R $ admits a complete PSC metric, does $ X $ admit a PSC metric?}}
\medskip

The noncompact cylinder $ M $ has two ends, with one-point compactification $ X \times \mathbb{S}^{1} $ and two-point compactification $ X \times [0, 1] $. In our recent work \cite{RX2}, \cite{XU10}, we showed that for every $ X $ with $ \dim X \geqslant 2 $, the $ \mathbb{S}^{1} $-stability conjecture holds by assuming some angle condition; we also showed that for every $ X $ with $ \dim X \geqslant 2 $, the existence of PSC metric with nonnegative mean curvature on $ X \times [0, 1] $ implies the existence of PSC metric on $ X $, provided that an analogous angle condition should hold. 

Inspired by the results of one- and two-points compactification of $ M $, it is natural to ask whether we can address the conjecture by imposing some version of angle condition also. Fix $ 0 \in \R $ and $ X_{0} : = X \times \lbrace 0 \rbrace \cong X $. Let $ \imath : X_{0} \rightarrow M $ be the natural inclusion. For the Riemannian manifold $ (M, g) $, we denote the scalar curvature on $ M $ by $ R_{g} $, the second fundamental form and mean curvature on $ X_{0} $ by $ A_{g} $ and $ h_{g} $, respectively. We say that $ (M, g) $ is of {\it{bounded geometry}} if $ g $ is a complete metric, having positive injectivity radius, and the curvature tensors and all their covariant derivatives are bounded. For the canonical tangent vector $ \partial_{\xi} \in \Gamma( T\R) $, we choose the unit normal vector field $ \nu_{g} $ along $ X_{0} $ such that
\begin{equation}\label{Intro:eqn1}
\angle_{g} (\nu_{g}, \partial_{\xi}) : = \cos^{-1} \left( \frac{g(\partial_{\xi}, \nu_{g})}{g(\nu_{g}, \nu_{g})^{\frac{1}{2}}g(\partial_{\xi}, \partial_{\xi})^{\frac{1}{2}}} \right) = \cos^{-1} \left( \frac{g(\partial_{\xi}, \nu_{g})}{g(\partial_{\xi}, \partial_{\xi})^{\frac{1}{2}}} \right) \in \left[0, \frac{\pi}{2}\right].
\end{equation}
Our first main result states as follows:
\begin{theorem}\label{intro:thm1}
    Let $ X $ be an oriented, closed manifold. Let $ (M = X \times \R, g) $ be a Riemannian noncompact cylinder of bounded geometry such that $ n : = \dim M \geqslant 3 $. If $ R_{g}  \geqslant \kappa_{0} > 0 $ on $ M $, and
    \begin{equation*}
        \angle_{g}(\nu_{g}, \partial_{\xi}) \in [0, \frac{\pi}{4})
    \end{equation*}
    along the hypersurface $ X_{0} $, then there exists a metric $ \tilde{g} $ in the conformal class of $ g $ such that $ \imath^{*} \tilde{g} $ is a PSC metric on $ X \cong X_{0} $.
\end{theorem}
\begin{remark}\label{intro:re1}It is worth mentioning that if replace the bounded geometry condition by assuming $ g $ is of {\it{bounded curvature}} in the sense of Aubin \cite[Chapter 2]{Aubin}, $ g $ has positive injectivity radius, and the sectional curvature of the complete metric $ g $ is uniformly bounded on $ X \times \mathbb{R} $, then the same conclusion of Theorem \ref{intro:thm1} holds, provided that all other assumptions in Theorem \ref{intro:thm1} keep unchanged.
\end{remark}

There is no difference if we replace the hypersurface $ X_{0} $ by any other hypersurfaces $ X_{\xi_{0}} : = X \times \lbrace \xi_{0} \rbrace, \forall \xi_{0} \in \R $. Therefore the same conclusion above holds on every $ X_{\xi} $ if we replace the angle condition by $ \angle_{g}(\nu_{g, \xi_{0}}, \partial_{\xi}) $ where $ \nu_{g, \xi_{0}} $ is the unit normal vector field on $ X_{\xi_{0}} $ defined by (\ref{Intro:eqn1}). It follows that the conjecture of Rosenberg and Stolz can be partially answered in the following sense:
\begin{corollary}
    Let $ X $ be a oriented closed manifold with $ \dim X \geqslant 2 $. If $ X $ admits no PSC metric, then for any Riemannian metric $ g $ on $ X \times \R $, either the $ g $-angle condition $ \angle_{g}(\nu_{g, \xi_{0}}, \partial_{\xi}) < \frac{\pi}{4} $ fails on all hypersurfaces $ X_{\xi_{0}}, \forall \xi_{0} \in \R $, or $ g $ fails to have bounded curvature/bounded geometry, or $ g $ cannot be a complete metric with uniform positive scalar curvature.
\end{corollary}

Given any complete smooth manifold $(N, g_{1}) $ (possibly with boundary), the conformal class of $ g_{1} $ is denoted by
\begin{equation*}
    [g_{1}] = \left\lbrace e^{2f} g_{1} : f \in \calC^{\infty}(N) \right\rbrace.
\end{equation*}
Since our construction of the new PSC metric on $ X $ is within the conformal class of the original metric, we show the transpose of the positivity of the Yamabe constants from $ X \times \R $ to $ X $.
\begin{corollary}\label{intro:cor2}
Let $ X $ be an oriented, closed manifold. Let $ (M = X \times \R, g) $ be a Riemannian noncompact cylinder of bounded geometry or bounded curvature such that $ n : = \dim M \geqslant 3 $. If $ R_{g} \geqslant \kappa_{0} > 0 $, then the Yamabe constant of the conformal class $ [g] $ on $ M $ is positive. If in addition,
    \begin{equation*}
        \angle_{g}(\nu_{g}, \partial_{\xi}) \in [0, \frac{\pi}{4})
    \end{equation*}
along the hypersurface $ X_{0} $, then the Yamabe constant of the conformal class $ [\imath^{*}g] $ on $ X $ is also positive.
\end{corollary}
The definition of the Yamabe constant on noncompact complete manifolds is given in \S2. Finally we can show that the uniform PSC condition can be dropped.

Generally, there are two types of methods in the study of scalar curvature problems on hypersurfaces: ones is based on Dirac operator and index theory on spin manifolds; the other is based on the geometric measure theory initiated by Schoen-Yau minimal hypersurfaces method \cite{SY2}. On compact manifolds, these methods were applied in \cite{CZ}, \cite{GROMOV2}, \cite{GL}, \cite{Rade},  \cite{Zeidler}. On noncompact cylinder $ M $, the results of index-theoretic obstructions to PSC metrics can be found in \cite{HPS}, \cite{JR}, \cite{RosSto}, \cite{Zeilder2}. A comprehensive study of the conjecture of Rosenberg and Stolz in dimensions 6 and 7 by $ \mu $-bubble approach can be found in \cite{CRZ}. For the application of $ \mu $-bubble to aspherical conjecture in dimensions 4 and 5 we refer to \cite{CL}, \cite{CLL}, \cite{GROMOV3}. 

Inspired by our work \cite{RX2}, \cite{XU10} on one-point and two-point compactifications of $ M $, we study this type of positive scalar curvature problem by a different method using conformal geometry and partial differential equations. Analogous to \cite{RX2} and \cite{XU10}, it is a codimension two approach by introducing an auxiliary one-dimensional space $ \mathbb{S}^{1} $. Precisely speaking, we look for a PSC metric on $ X_{0} \in M $ by considering a conformal factor on $ X \times \mathbb{R} \times \mathbb{S}^{1} $, which is initially constructed through an elliptic partial differential equation on the space $ X \times \mathbb{S}^{1} $. Our approach does not impose any topological restrictions on either $ X $ or $ M $, but instead imposes a conformally invariant geometric condition.

We assign the $ t $-variable to the circle $ \mathbb{S}^{1} $ with standard metric $ dt^{2} $. We will interchangeably use $ \mathbb{S}^{1} = \mathbb{S}_{t}^{1} $ and $ \R = \R_{\xi} $ for clarity. We always identify $ X \cong X \times \lbrace 0 \rbrace_{\xi} \subset M $ and $ W \cong W \times \lbrace 0 \rbrace_{\xi} \subset M \times \mathbb{S}^{1}_{t} $. Denote the Riemannian metric on $ M \times \mathbb{S}^{1} $ by $ \tilde{g} : = g \oplus dt^{2} $. We define $ \sigma : W \rightarrow M \times \mathbb{S}^{1}_{t} $ is given by $ \sigma(w) = (w, 0) $ with fixed point $ 0 \in \R_{\xi} $. Fix a point $ P \in \mathbb{S}^{1} $, we also define $ \tau_{1} : M \rightarrow M \times \mathbb{S}^{1}, \tau_{2} : X \rightarrow W $ by sending $ M \ni x \mapsto (x, P) = \tau_{1}(x), X \ni y \mapsto (y, P) = \tau_{2}(y) $, respectively. Under local parametrization, we may identify $ P $ with the point $ 0 $ in any local chart containing $ P $. We use the labels $ \lbrace 0 \rbrace_{\xi} \in \R_{\xi} $ and $ \lbrace 0 \rbrace_{t} \cong \lbrace P \rbrace \in \mathbb{S}_{t}^{1} $ to distinguish these two points. Our approach could be summarized as the following diagram:
\begin{equation}\label{diagram}
\begin{tikzcd} (M, g) \arrow[r,"\tau_{1}"] & (M \times \mathbb{S}^{1}_{t},g \oplus dt^{2} = \bar{g}) \\
(X, \imath^*g) \arrow[r,"\tau_{2}"]\arrow[u,"\imath"] &  (W = X \times \mathbb{S}^{1}_{t}, \sigma^*(\bar{g})) 
\arrow[u,"\sigma"]
\end{tikzcd}
\end{equation}
We next show that the transpose of the positivity of the Yamabe constants also hold on compact manifolds with non-empty smooth boundary $ (Z, \partial Z, g) $ by the same codmiensional two approach. More precisely, the diagram (\ref{diagram}) applies when replacing all $ X $ by $ (Z, \partial Z) $ also. The notion of the Yamabe constant for $ (Z, \partial Z, g) $, is defined in \S4 below.
\begin{theorem}\label{intro:thm2}
    Let $ (Z, \partial Z) $ be an oriented, compact manifold with smooth boundary. Let $ (M, \partial M, g) = (Z \times \R, \partial Z \times \R, g) $ be a Riemannian noncompact cylinder of bounded geometry such that $ n : = \dim M \geqslant 3 $. If $ (M, \partial M, g) $ has positive Yamabe constant with respect to the conformal class $ [g] $, and
    \begin{equation*}
        \angle_{g}(\nu_{g}, \partial_{\xi}) \in [0, \frac{\pi}{4}),
    \end{equation*}
    along the hypersurface $ (Z \times \lbrace 0 \rbrace, \partial Z \times \lbrace 0 \rbrace ) : = Z_{0} $, then the Yamabe constant of $ (Z, \partial Z) \cong Z_{0} $ with respect to the conformal class $ [\imath^{*}g] $ is also positive.
\end{theorem}
Again the same conclusion holds if we assume that $ (M, \partial M) $ has a complete metric $ g $ that is of {\it{bounded curvature}}.

It is straightforward that the same argument of Theorem \ref{intro:thm2} holds when $ \partial Z = \emptyset $ by removing the boundary condition of (\ref{BD:eqn1}) in the argument of Proposition \ref{BD:prop1}. Hence we can drop the uniform PSC assumption in Corollary \ref{intro:cor2}:
\begin{corollary}\label{intro:cor3}
Let $ X $ be an oriented, closed manifold. Let $ (M = X \times \R, g) $ be a Riemannian noncompact cylinder of bounded geometry or bounded curvature such that $ n : = \dim M \geqslant 3 $. If the Yamabe constant of the conformal class $ [g] $ on $ M $ is positive such that
    \begin{equation*}
        \angle_{g}(\nu_{g}, \partial_{\xi}) \in [0, \frac{\pi}{4})
    \end{equation*}
along the hypersurface $ X_{0} $, then the Yamabe constant of the conformal class $ [\imath^{*}g] $ on $ X $ is also positive.
\end{corollary}
\medskip

For $ \mathbb{S}^{1} $-stability conjecture \cite{RX2} and scalar and mean curvature comparison problem on compact cylinders \cite{XU10}, the positivity of the Yamabe constant is the same as the positivity of scalar curvature (and nonnegative mean curvature when having nonempty boundary). But they are not equivalent in noncompact setting. Therefore, beyond the scope of the Rosenberg-Slotz conjecture, we give a result by transposing the positivity of Yamabe constant on $ X \times \R $ to the positivity of the Yamabe constant on compact, oriented manifold $ X $ with $ \dim X \geqslant 2 $ (possibly with boundary), provided that some geometric conditions are satisfied. 

As an outline of the paper, we prove a series of technical results in \S2 when $ X $ is a closed manifold. The conformal factor we applied on $ X $ was constructed in Proposition \ref{Set:prop1}. In particular, Lemma \ref{Set:lemma1} implies that the ellipticity of the operator $ L $ in (\ref{Set:eqn3}) is determined by the $ g $-angle condition (\ref{Set:eqn0}). Lemma \ref{Set:lemma4} gives a partial $ \calC^{2} $-estimate for $ u_{0} $ , where global positivity of the metric $ g $ is essentially used. We also prove several technical lemmas which allow us to compare the Laplacians for $ u', \tilde{u}', u_{0} $ on $ M \times \mathbb{S}^{1} $, $ M $, and $ W $ respectively.

Our first main theorem with respect to the noncompact cylinder without boundary was given in Theorem \ref{PSC:thm1}. The technical results of \S2 appear in the Gauss-Codazzi equations and conformal transformations of Ricci and scalar curvature tensors in the proof of Theorem \ref{PSC:thm1}. Corollary \ref{POS:cor1} indicates that the positivity of Yamabe constants can be transposed from $ X \times \R $ to $ X $, provided that the $ g $-angle condition holds.

In \S4, we generalize our codimension two approach on noncompact cylinder with smooth boundary of cylindral type. Our second main theorem with respect to the noncompact cylinder with smooth boundary of cylindral type was given in Theorem \ref{BD:thm2}. This generalizes the transpose of the positivity of Yamabe constants from $ (X \times \R, \partial X \times \R, g) $ to $ (X, \partial X, \imath^{*}g) $. We introduce some technical lemmas that are analogous to those in \S2, but here for the manifolds with boundary. Eventually, Corollary \ref{BD:cor1} improves the statement of Corollary \ref{POS:cor1} by removing the uniform PSC assumption.

\section{The Elliptic Partial Differential Equation and The Positivity}
In this section, we give basic setup and all technical results that are essential for our main result in \S3. Some of these technical results are analogous to the technical results of \cite[Section 2]{RX2} and \cite[Section 2]{XU10} for compact manifolds. Recall that $ X $ is a closed, oriented manifold. But here $ M = X \times \R_{\xi} $ is a $ n $-dimensional noncompact cylinder with Riemannian metric $ g $. Therefore, we have to introduce different notions of positivity that are compatible with noncompact spaces and impose some geometric conditions on $ M $. We assume the familiarity of the standard knowledge of Sobolev space $ W^{k, p} $, and $ \calL^{p} $-elliptic regularity on compact manifolds, see e.g. \cite{Niren4}. For any metric $ g_{0} $, we set $ -\Delta_{g_{0}} $ to be the positive definite Laplace-Beltrami operator.

First of all, we introduce a global vector field defined on $ W = X \times \mathbb{S}_{t}^{1} $. Choosing local coordinates $ \lbrace x^{i}, x^{n} = \xi \rbrace $ around some point $ x \in X_{0} = X \times \lbrace 0 \rbrace_{\xi} \subset M $ with associated local frame $ \lbrace \partial_{i}, \partial_{n} = \partial_{\xi} \rbrace $ such that $ \lbrace \partial_{i} \rbrace_{i = 1}^{n - 1} $ are tangential to $ X_{0} $, we can write
\begin{equation*}
    \nu_{g} = a\partial_{\xi} + \sum_{i = 1}^{n - 1} b^{i} \partial_{i}.
\end{equation*}
Note that
\begin{equation*}
1 = g(\nu_{g}, \nu_{g}) = g(a\partial_{\xi}, \nu_{g}) + \sum_{i} g(b^{i} \partial_{i}, \nu_{g})  \Rightarrow a = g(\partial_{\xi}, \nu_{g})^{-1}.
\end{equation*}
$ a $ is nowhere vanishing since both $ \nu_{g} $ and $ \partial_{\xi} $ are nowhere tangent to $ X_{0} $. Clearly $ g(V, \nu_{g}) = 0 $. Therefore we have a $ g $-related global vector field $ V = \nu_{g} - a\partial_{\xi} $ on $ X_{0} $, which is locally expressed by $ V = \sum_{i = 1}^{n - 1} b^{i} \partial_{i} $. We extend $ V $ to $ W \cong X \times \lbrace 0 \rbrace _{\xi} \times \mathbb{S}^{1}_{t} $ via $ \left( \tau_{2} \right)_{*} $. Still denoted by $ V $ as a global vector field on $ W $, we have:
\begin{lemma}\label{Set:lemma1} If 
\begin{equation}\label{Set:eqn0}
\frac{g(\partial_{\xi},\partial_{\xi})}{g(\nu_{g},\partial_{\xi})^{2}}<
2 \; {\rm on} \; X_{0},
\end{equation}
 then 
the operator
\begin{equation*}
 L':=\nabla_{V}\nabla_{V} - \Delta_{\sigma^{*}\bar{g}}
\end{equation*}
is elliptic on $ W $.
\end{lemma}
\begin{proof} 
In any local coordinates,
 \begin{equation*}
    \nabla_{V}\nabla_{V} = \sum_{i, j = 1}^{n - 1} b^{i}(x) b^{j}(x) \frac{\partial^{2}}{\partial x^{i} \partial x^{j}} + G_{1}(x_{0})
\end{equation*}
where $ G_{1}(x) $ is a linear first order operator. In Riemannian normal coordinates $ (x^1,\ldots,x^{n-1}, x^{n + 1} = t) $ centered at a fixed $ x_0\in W $, we have
\begin{equation*}
    \Delta_{\sigma^{*}\bar{g}}|_{x_{0}} = \sum_{i = 1}^{n - 1} \frac{\partial^{2}}{\partial \left( x^{i} \right)^{2} } + \frac{\partial^{2}}{\partial t^{2}}.
\end{equation*}
The principal symbol of $L'$ at $ x_{0} $ is
\begin{equation*}
 \sigma_2(L')(x_0, \eta) =  - \sum_{i, j = 1}^{n - 1} b^{i}(x_{0})b^{j}(x_{0})
      \eta^{i} \eta^{j} 
     + \sum_{i = 1}^{n - 1} \left(\eta^{i} \right)^{2} + (\eta^{n + 1})^2.
\end{equation*}
Proving $ \sigma_2(L')(x_0, \eta)>0$ for $\eta \neq 0$ is equivalent to proving that
\begin{equation*}
    B : = - 
    \begin{pmatrix} (b^{1})^{2} & b^{1}b^{2} &\dotso & b^{1}b^{n - 1} \\ b^{2}b^{1} &( b^{2})^{2} & \dotso & b^{2}b^{n - 1} \\ \vdots & \vdots & \ddots & \vdots \\ b^{n - 1}b^{1} & b^{n - 1}b^{2} & \dotso & (b^{n - 1})^{2} \end{pmatrix} + I_{n - 1} : = B' + I_{n - 1}
\end{equation*}
is positive definite. The same argument of Sylvester criterion in \cite[Lemma 2.1]{RX2} then follows.
\end{proof}

We now introduce our partial differential equation, whose solution will be used to construct the conformal factors of the original metrics $ g $ and $ \bar{g} $ on $ M $, $ M \times \mathbb{S}^{1} $, respectively. With the standard $ dt^{2} $-metric on $ \mathbb{S}^{1}_{t} $, we fix some chart $ U \ni P $ with chart map $ \Psi : \mathbb{S}^{1} \supset U \rightarrow (-1, 1) $ such that $ \Psi(P) = 0 $ and still with locally $ t $-variable. We start with the construction of the inhomogeneous term of our partial differential equation.
\begin{lemma}\label{Set:lemma2} 
For any $ p \in \mathbb{N} $,  $ C \gg 1 $, and any positive $ \delta \ll 1 $, there exists a positive smooth function $F:W\to\R$ and small enough constant $ \epsilon \ll 1 $ such that $ F |_{X \times (-\frac{\epsilon}{2}, \frac{\epsilon}{2})_{t}} = C+1  $, $ F \equiv 0 $ outside $ X \times [-\epsilon, \epsilon]_{t} $ and $ \lVert  F \rVert_{\calL^{p}(W, \sigma^{*}\bar{g})} < \delta. $
\end{lemma}
\begin{proof}
Fix any $ p, C $ and $ \delta $. We take a constant function $ f = C + 1 : W \rightarrow \R $. There exists a nonnegative smooth function $  \phi :\mathbb{S}^{1}_{t} \rightarrow R $ such that 
\begin{equation*}
  \phi(t) \equiv 1, t \in \left(-\frac{\epsilon}{2}, \frac{\epsilon}{2}\right) \subset \Psi(U), \phi(t) = 0, t \notin [-\epsilon, \epsilon].
\end{equation*}
Set $ F = f \cdot (\Pi_{2}^{*} \phi) : W \rightarrow \R $ where $ \Pi_{2} : W \rightarrow \mathbb{S}_{t}^{1} $ is the natural projection. It follows that
\begin{equation*}
    \lVert F \rVert_{\calL^{p}(W, \sigma^{*}\bar{g})} < \delta
\end{equation*}
provided that $ \epsilon \ll 1 $ is small enough.
\end{proof}

We define the $\calC^{1, \alpha} $-norm on $ W $ for any $ \alpha \in (0, 1) $  by fixing a chart cover $ W = \bigcup_{i} (U_{i}, \varphi_{i}) $ such that
\begin{equation*}
    \lVert u \rVert_{\calC^{1, \alpha}(W)} = \lVert u \rVert_{\calC^{0}(W)} + \sup_{i, k} \lVert \partial_{x_{i}^{k}}u \rVert_{\calC^{0}(U_{i})} + \sup_{i, k} \sup_{x \neq x', x, x' \in U_{i}} \frac{\left\lvert \partial_{x_{i}^{k}}u(x) - \partial_{x_{i}^{k}}u(x') \right\rvert}{\lvert x - x' \rvert^{\alpha}}
\end{equation*}
where $ \partial_{x_{i}^{k}} u $ are local representations in $ U_{i} $ with local coordinates $ \lbrace x_{i}^{1}, \dotso, x_{i}^{n - 1}, x_{i}^{n + 1} = t \rbrace $ and associated local frames $ \lbrace \partial_{x_{i}^{k}} \rbrace $. 

Let $ R_{\bar{g}} $ be the scalar curvature with respect to the metric $ \tilde{g} $ on $ M \times \mathbb{S}^{1} $. We now introduce the partial differential equation whose solution has small enough $ \calC^{1, \alpha} $ norm on $ W $. 
\begin{proposition}\label{Set:prop1}
Let $ (M, g) $ be a noncompact cylinder of bounded geometry and $ R_{g} \geqslant \kappa_{0} > 0 $ for some $ \kappa_{0} \in \R $. Let $ (W, \sigma^{*} \bar{g}) $ be as above. Assume that (\ref{Set:eqn0}) holds. For any positive constant $ \kappa \ll 1 $, any positive constant $ C $, and any $ p > n = \dim W $, there exists an associated $ F $ and $ \delta $ in the sense of Lemma \ref{Set:lemma2}, such that the following partial differential equation     
\begin{equation}\label{Set:eqn3}
L u : = 4\nabla_{V} \nabla_{V} u - 4\Delta_{\sigma^{*} \bar{g}}  u + R_{\bar{g}} |_{W}u =
F  \; {\rm in} \; W
\end{equation}
admits a unique smooth solution $ u $ with
\begin{equation}\label{Set:eqn4}
\lVert u \rVert_{\calC^{1, \alpha}(W)} < \kappa
\end{equation}
for some $ \alpha \in (0, 1) $ such that $ \alpha \geqslant 1 - \frac{n}{p} $.
\end{proposition}
\begin{proof}
Since $ X $ is compact, (\ref{Set:eqn0}) implies that the maximum of the continuous function $ \frac{g(\partial_{\xi},\partial_{\xi})}{g(\nu_{g},\partial_{\xi})^{2}} $ is also strictly less than $ 2 $, and hence the uniform ellipticity of the operator $ L $ is obtained by Lemma \ref{Set:lemma1}. 

Fix $ \kappa \ll 1, C $, and $ p > n $. We then fix some $ \alpha \in (0, 1) $ such that $ 1 + \alpha \geqslant 2 - \frac{n}{p} $. Denote $ C = C(W, g, n, p, L) $ by the constant of $ \calL^{p} $ elliptic regularity estimates with respect to $ L $, and $ C' = C'(W, g, n, p, \alpha) $ by the constant of the Sobolev embedding $ W^{2, p} \hookrightarrow \calC^{1, \alpha} $. Fix $ \delta $ such that $ \delta C C' < \kappa $. Finally, we choose an associated $ F $ in the sense of Lemma \ref{Set:lemma2}.

By assumption, $ R_{g} \geqslant \kappa_{0} > 0 $ on $ M $ for some positive constant $ \kappa_{0} $, therefore $ R_{\bar{g}} > 0 $ on $ M \times \mathbb{S}^{1}_{t} $, hence $ R_{\bar{g}} |_{W} > 0 $ uniformly. By the maximum principle for closed manifolds, the elliptic operator $ L $ is an injective operator. It follows that there exists a unique solution $ u \in H^{1}(W, \sigma^{*}\bar{g}) $ of (\ref{Set:eqn3}) by Fredholm dichotomy. By standard $ H^{s} $-type elliptic theory and the smoothness of $ F $, $ u \in \calC^{\infty}(W) $. 

By standard $ \calL^{p} $-regularity theory \cite{Niren4}, 
\begin{equation*}
    \lVert u \rVert_{W^{2, p}(W, \sigma^{*} \bar{g})} \leqslant C \left( \lVert F \rVert_{\calL^{p}(W, \sigma^{*} \bar{g})} + \lVert u \rVert_{\calL^{p}(W, \sigma^{*}\bar{g})} \right).
\end{equation*}
Due to the injectivity of the operator, a very similar argument of \cite[Proposition 2.1]{RX2} shows that $ \lVert u \rVert_{\calL^{p}(W, \sigma^{*} \bar{g})} $ can be bounded above by $ \lVert Lu \rVert_{\calL^{p}(W, \sigma^{*}\bar{g})} $. It follows that the $ \calL^{p}$ estimates of our solution of (\ref{Set:eqn3}) can be improved by
\begin{equation*}
     \lVert u \rVert_{W^{2, p}(W, \sigma^{*} \bar{g})} \leqslant C \lVert F \rVert_{\calL^{p}(W, \sigma^{*} \bar{g})}.
\end{equation*}
By Sobolev embedding, it follows that
\begin{equation*}
    \lVert u \rVert_{\calC^{1, \alpha}(W)} \leqslant C' \lVert u \rVert_{W^{2, p}(W, \sigma^{*}\bar{g})} \leqslant CC'\lVert F \rVert_{\calL^{p}(W, \sigma^{*} \bar{g})} < CC' \delta < \kappa. 
\end{equation*}
\end{proof}

We now introduce the notions of positivity that we will use for noncompact manifolds. Such a global positivity is crucial to get partial $ \calC^{2} $-estimates of our candidate of conformal factor, as we shall see in Lemma \ref{Set:lemma4} below. 

Let $ (N, \hat{g}) $ be any noncompact manifold with some Riemannian metric $ \hat{g} $ and scalar curvature $ R_{\hat{g}} $, $ \dim N = n $. We define
\begin{equation*}
E_{\hat{g}}(v) = \frac{\frac{4(n -1)}{n - 2} \int_{N} \lvert \nabla_{\hat{g}} v \rvert^{2} \dvoll + \int_{N} R_{\hat{g}} v^{2} \dvoll}{\left( \int_{N} v^{\frac{2n}{n - 2}} \dvoll \right)^{\frac{n-2}{n}}}, \forall v \in \calC_{c}^{\infty}(N) \backslash \lbrace 0 \rbrace.
\end{equation*}
Analogous to the compact case, the {\it{Yamabe constant}} of the conformal class $ [\hat{g}] $ on $ N $ is defined by
\begin{equation}\label{Set:eqnY1}
    Y(N, \hat{g}) = \inf_{v \in \calC_{c}^{\infty}(N) \backslash \lbrace 0 \rbrace} E_{\hat{g}}(v).
\end{equation}
Let $ \lbrace K_{i} \rbrace_{i \in \mathbb{N}} $ be a compact exhaustion of $ N $. Due to Kim \cite{Kim1}, we can also define the {\it{Yamabe constant at infinity}} of $ (N, \hat{g}) $ for complete metric $ \hat{g} $ by
\begin{equation*}
    Y_{\infty}(N, \hat{g}) = \lim_{i \mapsto \infty} Y(N \backslash K_{i}, \hat{g}).
\end{equation*}
It is known \cite[Lemma 2.1]{Wei} that for any complete noncompact manifold $ (N, \hat{g}) $, we have
\begin{equation}\label{Set:eqn5}
    -\lVert (R_{\hat{g}})_{-} \rVert_{\calL^{\frac{n}{2}}(N, \hat{g})} \leqslant Y(N, \hat{g}) \leqslant Y_{\infty}(N, \hat{g}) \leqslant \frac{4(n - 1)}{n -2}\Lambda
\end{equation}
where $ (R_{\hat{g}})_{-} $ is the negative part of the scalar curvature on $ N $, and $ \Lambda $ is the best Sobolev constant on $ \R^{n} $.

In addition, we denote the infimum of the $ \calL^{2} $-spectrum of the conformal Laplacian on $ (N, \hat{g}) $ by
\begin{equation*}
    \mu(N, \hat{g}) = \inf_{v \in \calC_{c}^{\infty}(N) \backslash \lbrace 0 \rbrace} \frac{\frac{4(n -1)}{n - 2} \int_{N} \lvert \nabla_{\hat{g}} v \rvert^{2} \dvoll + \int_{N} R_{\hat{g}} v^{2} \dvoll}{\int_{N} v^{2} \dvoll}.
\end{equation*}
The next result is essentially due to Grosse \cite[Lemma 7]{Grosse}:
\begin{lemma}\label{Set:lemma3}
If $ (N, \hat{g}) $ is a noncompact Riemannian manifold of bounded geometry with $ R_{\hat{g}} \geqslant \kappa_{0} 
> 0 $ for some $ \kappa_{0} \in \R $, then $ \mu(N, \hat{g}) > 0 $. Consequently, $ Y(N, \hat{g}) > 0 $.
\end{lemma}
\begin{proof}
By $ R_{\hat{g}} \geqslant \kappa_{0} > 0 $, it follows that
\begin{align*}
    \frac{\frac{4(n -1)}{n - 2} \int_{N} \lvert \nabla_{\hat{g}} v \rvert^{2} \dvoll + \int_{N} R_{\hat{g}} v^{2} \dvoll}{ \int_{N} v^{2} \dvoll } & \geqslant \frac{\int_{N} R_{\hat{g}} v^{2} \dvol}{\int_{N} v^{2} \dvoll} \geqslant \kappa_{0} > 0, \forall v \in \calC_{c}^{\infty}(N) \backslash \lbrace 0 \rbrace.
\end{align*}
It follows that $ \mu(N, \hat{g}) > 0 $. By \cite[Theorem 3.18, Corollary 3.19]{Hebey}, the Sobolev embedding
\begin{equation*}
H^{1}(M, g) \hookrightarrow \calL^{\frac{2n}{n - 2}}(M, g)
\end{equation*}
is continuous, i.e. there exists a constant $ C = C(M, g) $ such that
\begin{equation*}
    \lVert v \rVert_{\calL^{2n}{n -2}(M, g)} \leqslant C \lVert v \rVert_{H^{1}(M, g)}, \forall v \in H^{1}(M, g).
\end{equation*}
Clearly $ Y(N, \hat{g}) \geqslant 0 $ by (\ref{Set:eqn5}) and the positivity of $ R_{\hat{g}} $. In contrast, assume that $ Y(N, \hat{g}) = 0 $. It follows that there exists a sequence $ \lbrace v_{i} \rbrace $ such that $ \lVert v_{i} \rVert_{\calL^{\frac{2n}{n - 2}}(N, \hat{g})} = 1, \forall i $ and
\begin{equation*}
    \frac{4(n -1)}{n - 2} \int_{N} \lvert \nabla_{\hat{g}} v_{i} \rvert^{2} \dvoll + \int_{N} R_{\hat{g}} v_{i}^{2} \dvoll \xrightarrow{i \rightarrow \infty} 0.
\end{equation*}
Since $ \mu(N, \hat{g}) \geqslant \kappa_{0} > 0 $, it follows that
\begin{equation*}
\lVert v_{i} \rVert_{\calL^{2}(N, \hat{g})} \xrightarrow{i \rightarrow \infty} 0 \Rightarrow \lVert \nabla_{\hat{g}} v_{i} \rVert_{\calL^{2}(N, \hat{g})} \xrightarrow{i \rightarrow \infty} 0.
\end{equation*}
But by continuity of the Sobolev embedding, we observe that for large enough $ i $,
\begin{equation*}
1 = \lVert v_{i} \rVert_{\calL^{\frac{2n}{n - 2}}(N, \hat{g})} \leqslant C \lVert v_{i} \rVert_{H^{1}(N, \hat{g})} \ll 1
\end{equation*}
which reaches a contradiction. Therefore we must have $ Y(N, \hat{g}) > 0 $.
\end{proof}
\begin{remark}\label{Set:re0}
The hypothesis that $ \hat{g} $ is of bounded geometry can be replaced by requiring that $ \hat{g} $ is a complete metric, has positive injectivity radius, and uniformly bounded sectional curvature. With these hypotheses, the Sobolev embedding $ H^{1}(N, \hat{g}) \rightarrow \calL^{\frac{2n}{n - 2}}(N, \hat{g}) $ is also continuous \cite[Theorem 2.21]{Aubin}.
\end{remark}
We now introduce the conformal factor that will be applied to metrics on $ M \times \mathbb{S}^{1}_{t} $, $ M $ and $ X $, respectively. To begin with, we define a new function
\begin{equation*}
    u_{0} : = u + 1, u_{0} : W \rightarrow \R
\end{equation*}
where $ u $ is the solution of (\ref{Set:eqn3}). By (\ref{Set:eqn4}), $ u_{0} > 0 $ for small enough $ \kappa $. With the natural projection $ \pi : M \times \mathbb{S}^{1}_{t} \rightarrow W $, we can pullback $ u_{0} $ to $ M \times \mathbb{S}_{t}^{1} $ by $ u_{1} : = (\pi)^{*} u_{0} $. We also pullback $ u $ to $ M \times \mathbb{S}_{t}^{1} $ by $ \bar{u}_{1} : = (\pi)^{*} u $ for later use. Fixing a sufficiently large closed interval $ I_{\xi} = [-k, k] $ such that $ [-1, 1] \subset I_{\xi} \in \R $, we define a smooth, nonnegative function
\begin{equation}\label{Set:eqnC0}
    \varphi : \R_{\xi} \rightarrow \R, \varphi \geqslant 0 \; {\rm on} \; R_{\xi}, \varphi(\xi) = 1 \; {\rm on} \; [-1, 1], \varphi(\xi) \equiv 0 \; {\rm on} \; \R_{\xi} \backslash I_{\xi}.
\end{equation}
With the natural projection $ \Pi_{1} : M \times \mathbb{S}_{t}^{1} \rightarrow \R_{\xi} $, we obtain the pullback function
\begin{equation}\label{Set:eqnC1}
    \Phi:= \Pi_{1}^{*}\varphi : M \times \mathbb{S}_{t}^{1} \rightarrow \R
\end{equation}
such that $ \Phi \geqslant 0 $ on $ M \times \mathbb{S}_{t}^{1} $, $ \text{supp}(\Phi) \in X \times \mathbb{S}_{t}^{1} \times I_{\xi} $, and $ \Phi = 1 $ on each hypersurface $ W \times \lbrace \xi \rbrace = X \times \mathbb{S}_{t}^{1} \times \lbrace \xi \rbrace, \xi \in [-1, 1] $. Define
\begin{equation}\label{Set:eqn6}
u' : = u_{1} : M \times \mathbb{S}_{t}^{1} \rightarrow \R, \bar{u}' = \Phi \cdot \bar{u}_{1} : M \times \mathbb{S}_{t}^{1} \rightarrow \R, u' > 0, \bar{u}' \geqslant 0 \; {\rm on} \; M \times \mathbb{S}_{t}^{1}.
\end{equation}
The function $ (u')^{\frac{4}{n - 2}} $ will be our choice of conformal factor on $ M \times \mathbb{S}_{t}^{1} $. It follows that
\begin{equation}\label{Set:eqn7}
u' |_{W \times [-1, 1]_{\xi}} =u_{1} = \bar{u}_{1} + 1 = \bar{u}' |_{W \times [-1, 1]_{\xi}} + 1, \partial_{\xi}^{l} u' = 0 \; {\rm on} \; W \times [-1, 1]_{\xi} \subset M \times \mathbb{S}^{1}_{t}, \forall l \in \mathbb{Z}_{> 0}.
\end{equation}
Denote $ \tilde{u}' : = \tau_{1}^{*} u' : M \rightarrow \R $. Following our diagram (\ref{diagram}), we have
\begin{equation*}
   \sigma^{*}u' = u' |_{W} = u_{1} |_{W} = u_{0}, \imath^{*} \tilde{u}' = \imath^{*} \tau_{1}^{*} u' = \tau_{2}^{*} \sigma^{*} u' = \tau_{2}^{*} u_{0} = u_{0} |_{X \times \lbrace 0 \rbrace_{\xi} \times \lbrace P \rbrace_{t}}.
\end{equation*}
Clearly all those functions defined above are positive functions on corresponding spaces. The relations above hold analogously for $ \bar{u} $ also.
\medskip

Analogous to the compact cases, we need to apply the positivity of the Yamabe constant define in (\ref{Set:eqnY1}) to get partial $ \calC^{2} $-estimate for the solution $ u $ of (\ref{Set:eqn4}) within a tubular b of $ X \cong X \times \lbrace 0 \rbrace_{\xi} $ in $ W $, i.e. we need to estimate $ \left\lVert \frac{\partial^{2} u}{\partial t^{2}} \right\rVert_{\calC^{0}\left(X \times \lbrace 0 \rbrace_{\xi} \times \left( -\frac{\epsilon}{4}, \frac{\epsilon}{4} \right)_{t}\right)} $, which implies the same partial $ \calC^{2} $-estimate for $ u_{0} $ within the region $ X \times [-1, 1]_{\xi} \times \left( -\frac{\epsilon}{4}, \frac{\epsilon}{4} \right)_{t} $, and hence $ u' $ within the region $ X \times [-1, 1]_{\xi} \times \lbrace P \rbrace_{t} $ also.
\begin{lemma}\label{Set:lemma4}
Choosing $ \epsilon $ defined in Lemma \ref{Set:lemma2} to be small enough that will be determined below. Assume that $ (M, g) $ is a noncompact cylinder of bounded geometry such that $ R_{g} \geqslant \kappa_{0} > 0 $ for some $ \kappa_{0} \in \R $. Let $ (W, \sigma^{*} \bar{g}) $ be the associated space defined in (\ref{diagram}). Let $ p, C, \delta $ and $ F $ be the same as in Proposition \ref{Set:prop1}. If $ u $ is the associated solution of (\ref{Set:eqn3}), then there exists a constant $ \eta' \ll 1 $ such that
\begin{equation}\label{Set:eqn8}
    \left\lVert \frac{\partial^{2} u}{\partial t^{2}} \right\rVert_{\calC^{0}\left(X \times \lbrace 0 \rbrace_{\xi} \times \left(-\frac{\epsilon}{4}, \frac{\epsilon}{4} \right)_{t} \right)} < \eta'.
\end{equation}
\end{lemma}
The proof is analogous to the proof of \cite[Lemma 2.3]{XU10}. The essential difference is that we introduce a cut-off function $ \Phi $ to take advantage of the positivity of the Yamabe constant. We cannot associate $ u' $ directly to $ Y(N, g) $ on $ M $ since $ u' $ is not compactly supported.
\begin{proof}
Set $ U_{1} = X \times \lbrace 0 \rbrace_{\xi} \times \left(-\frac{\epsilon}{4}, \frac{\epsilon}{4} \right)_{t} $, $ U_{2} =  X \times \lbrace 0 \rbrace_{\xi} \times \left(-\frac{\epsilon}{2}, \frac{\epsilon}{2} \right)_{t} $, $ O_{1} =  X \times \lbrace 0 \rbrace_{\xi} \times \left(-\frac{1}{4}, \frac{1}{4} \right) $, $ O_{2} =  X \times \lbrace 0 \rbrace_{\xi} \times \left(-\frac{1}{2}, \frac{1}{2} \right) $, $ O_{3} =  X \times \lbrace 0 \rbrace_{\xi} \times \left(-1, 1 \right) $. Since $ \bar{g} = g \oplus dt^{2} $, the scalar curvature $ R_{\bar{g}} $ is constant along $ t $-direction, in addition, the vector field $ V $ and $ \Delta_{\sigma^{*} \bar{g}} $ are constant on each $ t $-fiber. By Lemma \ref{Set:lemma2}, $ F \equiv C + 1 $ on $ U_{2} $. We apply $ \frac{\partial^{2}}{\partial t^{2}} $ on both sides of (\ref{Set:eqn3}) in $ U_{2} $, 
\begin{equation*}
  L_{\sigma^{*} \bar{g}} \left( \frac{\partial^{2}u}{\partial t^{2}} \right) : =  4\nabla_{V} \nabla_{V} \left( \frac{\partial^{2}u}{\partial t^{2}} \right) - 4\Delta_{\sigma^{*} \bar{g}}  \left( \frac{\partial^{2}u}{\partial t^{2}} \right) + R_{\bar{g}} |_{W}\left( \frac{\partial^{2}u}{\partial t^{2}} \right) = 0  \; {\rm in} \; U_{2}.
\end{equation*}
Set $ t = \epsilon t' $, the new $ t' $-variable is associated with the metric $ (dt')^{2} = \epsilon^{-2} dt^{2} $ on $ \mathbb{S}^{1} $. We have
\begin{equation*}
   (\epsilon^{-1})^{2} \left( 4\nabla_{V} \nabla_{V} \left( \frac{\partial^{2}u}{\partial (t')^{2}} \right) - 4\Delta_{\sigma^{*} \bar{g}}  \left( \frac{\partial^{2}u}{\partial (t')^{2}} \right) + R_{\bar{g}} |_{W}\left( \frac{\partial^{2}u}{\partial (t')^{2}} \right) \right)= 0  \; {\rm in} \; O_{2}.
\end{equation*}
Equivalently, the following PDE holds in $ O_{2} $ with the new metric $ \bar{g} \mapsto \epsilon^{-2} \bar{g} $,
\begin{equation}\label{Set:eqn9}
L_{\sigma^{*}(\epsilon^{-2}\bar{g})}\left( \frac{\partial^{2}u}{\partial (t')^{2}} \right) : = 4\nabla_{V_{\epsilon^{-2}\bar{g}}} \nabla_{V_{\epsilon^{-2}\bar{g}}} \left( \frac{\partial^{2}u}{\partial (t')^{2}} \right) - 4\Delta_{\sigma^{*} (\epsilon^{-2}\bar{g})}  \left( \frac{\partial^{2}u}{\partial (t')^{2}} \right) + R_{\epsilon^{-2}\bar{g}} |_{W}\left( \frac{\partial^{2}u}{\partial (t')^{2}} \right) = 0.
\end{equation}
Here $ V_{\epsilon^{-2}\bar{g}} $ is the vector field with respect to the new metric $ \epsilon^{-2} \bar{g} $ satisfying $ V_{\epsilon^{-2}\bar{g}} = \epsilon^{-1} V $. With the new metric $ \epsilon^{-2} \bar{g} $, we estimate $ \lVert \frac{\partial^{2}u}{\partial (t')^{2}} \rVert_{\calC^{0}(O_{1})} $ with local $ H^{s} $-type elliptic regularity. We set $ s > 0 $ such that $ s - \frac{n}{2} = 1 + \alpha \geqslant 2 - \frac{n}{p} $, where $ p $ and $ \alpha $ are given in Proposition \ref{Set:prop1}. With usual metric $ \sigma^{*} \bar{g} $, the standard $ H^{s} $-type local elliptic regularity estimates for second order elliptic operator $ L_{\sigma^{*}\bar{g}} $ gives
\begin{equation*}
\lVert v \rVert_{H^{s}(U, \sigma^{*} \bar{g})} \leqslant D_{s} \left( \lVert L_{\sigma^{*} \bar{g}} u \rVert_{H^{s - 2}(V, \sigma^{*} \bar{g})} + \lVert v \rVert_{\calL^{2}(V, \sigma^{*} \bar{g})} \right), \forall v \in \calC^{\infty}(V).
\end{equation*}
where the constant $ D_{s} $ only depends on the operator $ L_{\sigma^{*}\bar{g}} $, the metric, the degree $ s $ and the domain $ U, V $. For the scaled metric $ \sigma^{*} (\epsilon^{-2} \bar{g}) $, the second order elliptic operator $ L_{\sigma^{*}(\epsilon^{-2}\bar{g})} $ satisfies $ L_{\sigma^{*}(\epsilon^{-2}\bar{g})} = \epsilon^{2} L_{\sigma^{*}\bar{g}}$. We apply the local elliptic estimate in $ O_{1} $ for the second order elliptic operator $ L_{\sigma^{*}(\epsilon^{-2}\bar{g})} $ in (\ref{Set:eqn9}), 
\begin{equation}\label{Set:eqn70a}
\begin{split}
\left\lVert \frac{\partial^{2}u}{\partial (t')^{2}} \right\rVert_{H^{s}(O_{1}, \sigma^{*}(\epsilon^{-2} \bar{g}))} & \leqslant \epsilon^{-\frac{n}{2}} \left\lVert \frac{\partial^{2}u}{\partial (t')^{2}} \right\rVert_{H^{s}(O_{1}, \sigma^{*} \bar{g})} \\
& \leqslant \epsilon^{-\frac{n}{2}} \left( D_{s} \left\lVert L_{\sigma^{*}\bar{g}} \left( \frac{\partial^{2}u}{\partial (t')^{2}} \right) \right\rVert_{H^{s - 2}(O_{2}, \sigma^{*} \bar{g})} + \left\lVert \frac{\partial^{2}u}{\partial (t')^{2}} \right\rVert_{\calL^{2}(O_{2}, \sigma^{*} \bar{g})} \right) \\
& \leqslant D_{s, \epsilon}  \left\lVert L_{\sigma^{*}(\epsilon^{-2}\bar{g})} \left( \frac{\partial^{2}u}{\partial (t')^{2}} \right) \right\rVert_{H^{s - 2}(O_{2}, \sigma^{*} (\epsilon^{-2}\bar{g}))} + \epsilon^{-\frac{n}{2}} D_{s} \left\lVert \frac{\partial^{2}u}{\partial (t')^{2}} \right\rVert_{\calL^{2}(O_{2}, \sigma^{*} \bar{g})} \\
& =  D_{s} \left\lVert \frac{\partial^{2}u}{\partial (t')^{2}} \right\rVert_{\calL^{2}(O_{2}, \sigma^{*}(\epsilon^{-2} \bar{g}))}.
\end{split}
\end{equation}
With local $ H^{2} $-type local elliptic estimate for (\ref{Set:eqn3}) and H\"older's inequality,
\begin{equation}\label{Set:eqn10}
    \begin{split}
        \left\lVert \frac{\partial^{2}u}{\partial (t')^{2}} \right\rVert_{\calL^{2}(O_{2}, \sigma^{*} (\epsilon^{-2} \bar{g}))} & = \epsilon^{2 - \frac{n}{2}} \left\lVert \frac{\partial^{2}u}{\partial t^{2}} \right\rVert_{\calL^{2}(O_{2}, \sigma^{*} \bar{g})} \leqslant  \epsilon^{2 - \frac{n}{2}} \lVert u \rVert_{H^{2}(O_{2}, \sigma^{*}\bar{g})} \\
        & \leqslant \epsilon^{2 - \frac{n}{2}} D_{2} \left( \lVert F \rVert_{\calL^{2}(O_{3}, \sigma^{*}\bar{g})} + \lVert u \rVert_{\calL^{2}(O_{3}, \sigma^{*}\bar{g})} \right) \\
        & \leqslant \epsilon^{2 - \frac{n}{2}}D_{2}\lVert F \rVert_{\calL^{2}(O_{3}, \sigma^{*}\bar{g})} + \epsilon^{2 - \frac{n}{2}}D_{1} D_{2}\lVert u \rVert_{\calL^{\frac{2(n + 1)}{n - 1}}(O_{3}, \sigma^{*}\bar{g})} \\
        & \leqslant \epsilon^{2} D_{2}  \lVert F \rVert_{\calL^{2}(O_{3}, \sigma^{*}(\epsilon^{-2}\bar{g}))} + \epsilon^{2 - \frac{n}{n + 1}} D_{1}D_{2} \lVert u \rVert_{\calL^{\frac{2(n + 1)}{n - 1}}(O_{3}, \sigma^{*} (\epsilon^{-2} \bar{g}))} \\
        & \leqslant \epsilon^{2} D_{2}  \lVert F \rVert_{\calL^{2}(W, \sigma^{*}(\epsilon^{-2}\bar{g}))} + \epsilon^{2 - \frac{n}{n + 1}} D_{1}D_{2} \lVert u \rVert_{\calL^{\frac{2(n + 1)}{n - 1}}(W, \sigma^{*} (\epsilon^{-2} \bar{g}))}.
    \end{split}
\end{equation}
By Sobolev embedding inequality with respect to $ \sigma^{*} (\epsilon^{-2} \bar{g}) $,
\begin{equation}\label{Set:eqn11}
\left\lVert \frac{\partial^{2}u}{\partial (t')^{2}} \right\rVert_{\calC^{0}(O_{1})} \leqslant D_{0} \epsilon^{\frac{n}{2}} \left\lVert \frac{\partial^{2}u}{\partial (t')^{2}} \right\rVert_{H^{s}(O_{1}, \sigma^{*}(\epsilon^{-2} \bar{g}))}.
\end{equation}
Here $ D_{0}, D_{1}, D_{2} $ are independent of $ \epsilon $ and $ u $. It is well-known that the Yamabe constant (\ref{Set:eqnY1}) is invariant under conformal transformation, so is under the scaling of the metrics. By (\ref{Set:eqn5}) and Lemma \ref{Set:lemma3}), $ Y(M \times \mathbb{S}_{t}^{1}, \bar{g}) $ is finite and bounded below by some positive constant since $ \bar{g} $ has uniformly positive scalar curvature. According to the definition of the Yamabe constant $ \lambda : = Y(M \times \mathbb{S}_{t}^{1}, \bar{g})$, we have
\begin{align*}
    & \lVert v \rVert_{\calL^{\frac{2(n + 1)}{n - 1}}(M \times \mathbb{S}^{1}, \epsilon^{-2}\bar{g})}^{2} \\
    & \qquad \leqslant \lambda^{-1} \left( \frac{4n}{n - 1} \lVert \nabla_{\epsilon^{-2}\bar{g}} v \rVert_{\calL^{2}(M \times \mathbb{S}^{1}, \epsilon^{-2}\bar{g})}^{2} + \int_{M \times \mathbb{S}^{1}} R_{\epsilon^{-2}\bar{g}} v^{2} d\text{Vol}_{\epsilon^{-2}\bar{g}} \right), \forall v \in \calC_{c}^{\infty}(M \times \mathbb{S}^{1}) \backslash \lbrace 0 \rbrace. 
\end{align*}
Recall that $ X_{\xi} : = X \times \lbrace \xi \rbrace \subset M $. For every $ \xi \in \R_{\xi} $, we denote the natural inclusion by $ \sigma_{\xi} : X_{\xi} \times \mathbb{S}^{1} \rightarrow M \times \mathbb{S}^{1} $, and the space $ X_{\xi} \times \mathbb{S}^{1} $ by $ W_{\xi} $. Since $ M $ is of bounded geometry and $ \bar{g} $ is a product metric, $ (M \times \mathbb{S}^{1}, \bar{g}) $ is of bounded geometry. It implies that there exist two positive constants $ D_{3} $ and $ D_{4} $ such that two metrics $ \bar{g} $ and $ \sigma_{\xi}^{*} \bar{g} \oplus d\xi^{2} $ satisfy the pointwise relation
\begin{equation*}
   D_{4}^{-2} \sqrt{\det{\bar{g}}} \leqslant \sqrt{\det(\sigma_{\xi}^{*} \bar{g} \oplus d\xi^{2})} \leqslant D_{3}^{\frac{2(n + 1)}{n - 1}} \sqrt{\det{\bar{g}}}, \forall \xi.
\end{equation*}
Set $ \xi' = \epsilon^{-1} \xi $, we have
\begin{equation*}
   D_{4}^{-2} \sqrt{\det{(\epsilon^{-2}\bar{g})}} \leqslant \sqrt{\det(\sigma_{\xi}^{*}(\epsilon^{-2} \bar{g}) \oplus d(\xi')^{2})} \leqslant D_{3}^{\frac{2(n + 1)}{n - 1}} \sqrt{\det{(\epsilon^{-2}\bar{g})}}, \forall \xi.
\end{equation*}
Applying the three inequalities above and the construction of $ \Phi $ in (\ref{Set:eqnC1}), we have
\begin{equation}\label{Set:A1}
\begin{split}
& \lVert u \rVert_{\calL^{\frac{2(n + 1)}{n - 1}}(W, \sigma^{*} (\epsilon^{-2} \bar{g}))} \\
& \qquad = (2\epsilon)^{\frac{n -1}{2(n + 1)}} \left( \int_{-\epsilon^{-1}}^{\epsilon^{-1}} \int_{W} \lvert \bar{u}_{1} \rvert^{\frac{2(n + 1)}{n - 1}} d\text{Vol}_{\sigma^{*} (\epsilon^{-2} \bar{g})} d\xi'  \right)^{\frac{n - 1}{2(n + 1)}} \\
& \qquad = (2\epsilon)^{\frac{n -1}{2(n + 1)}} \left( \int_{W \times [-\epsilon^{-1}, \epsilon^{-1}]} \lvert \bar{u}_{1} \rvert^{\frac{2(n + 1)}{n - 1}} d\text{Vol}_{\sigma^{*} \bar{g} \oplus d(\xi')^{2}}  \right)^{\frac{n - 1}{2(n + 1)}} \\
& \qquad = (2\epsilon)^{\frac{n -1}{2(n + 1)}} \left( \int_{W \times [-\epsilon^{-1}, \epsilon^{-1}]} \lvert \Phi(\cdot, \xi') \cdot \bar{u}_{1} \rvert^{\frac{2(n + 1)}{n - 1}} d\text{Vol}_{\sigma^{*} \bar{g} \oplus d(\xi')^{2}}  \right)^{\frac{n - 1}{2(n + 1)}}  \\
& \qquad \leqslant (2\epsilon)^{\frac{n -1}{2(n + 1)}} D_{3}  \left( \int_{M \times \mathbb{S}^{1}_{t'}}  \lvert \bar{u} \rvert^{\frac{2(n + 1)}{n - 1}} d\text{Vol}_{\epsilon^{-2}\bar{g}} \right)^{\frac{n - 1}{2(n + 1)}} \\
& \qquad \leqslant (2\epsilon)^{\frac{n -1}{2(n + 1)} } D_{3} \lambda^{-\frac{1}{2}} \left( \frac{4n}{n - 1} \lVert \nabla_{\epsilon^{-2}\bar{g}} \bar{u} \rVert_{\calL^{2}(M \times \mathbb{S}^{1}_{t'}, \epsilon^{-2}\bar{g})} + \int_{M \times \mathbb{S}^{1}_{t'}} R_{\epsilon^{-2}\bar{g}} \bar{u}^{2} d\text{Vol}_{\epsilon^{-2}\bar{g}} \right)^{\frac{1}{2}} \\
& \qquad \leqslant (2\epsilon)^{\frac{n -1}{2(n + 1)}} D_{3} D_{4} \lambda^{-\frac{1}{2}} \left( \frac{4n}{n - 1}\int_{\R} \int_{W_{\xi}} \lvert \nabla_{\sigma_{\xi}^{*}(\epsilon^{-2}\bar{g})} \bar{u} \rvert^{2} d\text{Vol}_{\sigma_{\xi}^{*}(\epsilon^{-2}\bar{g})} d\xi' \right)^{\frac{1}{2}} \\
& \qquad \qquad + (2\epsilon)^{\frac{n -1}{2(n + 1)}} D_{3} D_{4} \lambda^{-\frac{1}{2}} \left( \int_{\R} \int_{W_{\xi}} R_{\epsilon^{-2}\bar{g}} \bar{u}^{2} d\text{Vol}_{\sigma^{*} (\epsilon^{-2}\bar{g})} d\xi' \right)^{\frac{1}{2}} \\
& = 2^{\frac{n - 1}{2(n + 1)}} \epsilon^{-\frac{1}{n + 1}} D_{3} D_{4} \lambda^{-\frac{1}{2}} \left( \frac{4n}{n - 1}\int_{\R} \int_{W_{\xi}} \lvert \nabla_{\sigma_{\xi}^{*}(\epsilon^{-2}\bar{g})} (\Phi(\cdot, \xi) \bar{u}_{1}) \rvert^{2} d\text{Vol}_{\sigma_{\xi}^{*}(\epsilon^{-2}\bar{g})} d\xi \right)^{\frac{1}{2}} \\
& \qquad \qquad + 2^{\frac{n - 1}{2(n + 1)}} \epsilon^{-\frac{1}{n + 1}} D_{3} D_{4} \lambda^{-\frac{1}{2}} \left( \int_{\R} \int_{W_{\xi}} R_{\epsilon^{-2}\bar{g}} (\Phi(\cdot, \xi) \bar{u}_{1})^{2} d\text{Vol}_{\sigma^{*} (\epsilon^{-2}\bar{g})} d\xi \right)^{\frac{1}{2}} \\
& \qquad \leqslant 2D_{3} D_{4}D_{\Phi} \lambda^{-\frac{1}{2}} \epsilon^{-\frac{1}{n + 1}} \times \\
& \qquad \qquad \times \left( \frac{4n}{n - 1} \max_{\xi \in I_{\xi}} \left( \lVert \nabla_{\sigma_{\xi}^{*} (\epsilon^{-2} \bar{g})}\bar{u}_{1} \rVert_{\calL^{2}(W_{\xi}, \sigma_{\xi}^{*}(\epsilon^{-2}\bar{g}))}^{2} \right) + \epsilon^{2} \max_{W} \lvert R_{\bar{g}} \rvert \lVert u \rVert_{\calL^{2}(W, \sigma^{*}(\epsilon^{-2}\bar{g}))}^{2} \right)^{\frac{1}{2}}
\end{split}
\end{equation}
Here $ D_{\Phi} $ depends on $ \Phi $ and the support $ I_{\xi} $, and is independent of $ \epsilon $ and $ u $. For clarity, we absorb the constant $ 2^{\frac{n - 1}{2(n + 1)}} $ into $ D_{\Phi} $. In this long derivation, we must use the the noncompact version of the Yamabe constant, since after scaling $ \xi \mapsto \xi' $, the cut-off function $ \Phi(\cdot, \xi') $, although is still compactly supported, may have support beyond a fixed compact set. 

Note that there are trivial diffeomorphisms between $ W $ and $ W_{\xi} $ for every $ \xi \in I_{\xi} = [-k, k] $. Since $ \bar{u}_{1} $ is constant for each $ \xi $-fiber, compactness of $ W \times I_{\xi} $ and smallness of $ \calC^{1, \alpha} $-norm of $ u $, we have $ \lVert \nabla_{\sigma_{\xi}^{*}\bar{g}} \bar{u}_{1} \rVert_{\calL^{2}(W_{\xi}, \sigma_{\xi}^{*} \bar{g})} \leqslant D_{5} \eta $ for some $ D_{5} $ independent of $ \bar{u}_{1} $ and $ \epsilon $. Without loss of generality, we may assume $ \max_{W} \lvert R_{\bar{g}} \rvert \leqslant D_{5} $ by increasing $ D_{5} $ if necessary. Note also that by Lemma \ref{Set:lemma2}, $ \lVert F \rVert_{\calL^{p}(W, \sigma^{*}\bar{g})} $ is small in the sense that $ (C + 1)^{p} \epsilon \ll 1 $, hence $ (C + 1)^{2} \epsilon \ll 1 $ since $ p > n \geqslant 3 $. Combining (\ref{Set:eqn4}), (\ref{Set:eqn70a}), (\ref{Set:eqn10}), (\ref{Set:eqn11}) with the above inequality, we have
\begin{equation}\label{Set:A2}
\begin{split}
\left\lVert \frac{\partial^{2}u}{\partial (t')^{2}} \right\rVert_{\calC^{0}(O_{1})} & \leqslant D_{0}\epsilon^{\frac{n}{2}} \lVert \frac{\partial^{2}u}{\partial (t')^{2}} \rVert_{H^{s}(O_{1}, \sigma^{*}(\epsilon^{-2} \bar{g}))} \leqslant D_{0} D_{s} \epsilon^{\frac{n}{2}} \left\lVert \frac{\partial^{2}u}{\partial (t')^{2}} \right\rVert_{\calL^{2}(O_{2}, \sigma^{*}(\epsilon^{-2} \bar{g}))} \\
& \leqslant \epsilon^{2 + \frac{n}{2}} D_{0}D_{2}D_{s}  \lVert F \rVert_{\calL^{2}(W, \sigma^{*}(\epsilon^{-2}g))} + \epsilon^{2 - \frac{n}{n + 1} + \frac{n}{2}} D_{0}D_{1}D_{2}D_{s} \lVert u \rVert_{\calL^{\frac{2(n + 1)}{n - 1}}(W, \sigma^{*} (\epsilon^{-2} \bar{g}))} \\
& \leqslant 2D_{0} D_{1} D_{2} D_{3}D_{4} D_{s} D_{\Phi} \lambda^{-\frac{1}{2}}  \epsilon^{1 + \frac{n}{2}} \times \\
& \qquad \times \left( \frac{4n}{n - 1} \max_{\xi \in I_{\xi}} \left( \lVert \nabla_{\sigma_{\xi}^{*} (\epsilon^{-2} \bar{g})} \bar{u} \rVert_{\calL^{2}(W_{\xi}, \sigma_{\xi}^{*}(\epsilon^{-2}\bar{g}))}^{2} \right) + \epsilon^{2} \max_{W} \lvert R_{\bar{g}} \rvert \lVert u \rVert_{\calL^{2}(W, \sigma^{*}(\epsilon^{-2}\bar{g}))}^{2} \right)^{\frac{1}{2}} \\
& \qquad \qquad + \epsilon^{2 + \frac{n}{2}} D_{0}D_{2}D_{s}  \lVert F \rVert_{\calL^{2}(W, \sigma^{*}(\epsilon^{-2}\bar{g}))} \\
& : = \bar{D}_{0}  \epsilon^{2} \left( \frac{4n}{n - 1} \max_{\xi \in I_{\xi}} \left( \lVert \nabla_{\sigma_{\xi}^{*}  \bar{g}} \bar{u} \rVert_{\calL^{2}(W_{\xi}, \sigma_{\xi}^{*}\bar{g})}^{2} \right) + \max_{W} \lvert R_{\bar{g}} \rvert \lVert u \rVert_{\calL^{2}(W, \sigma^{*}\bar{g})}^{2} \right)^{\frac{1}{2}} \\
& \qquad + \epsilon^{2 + \frac{n}{2}} D_{0}D_{2}D_{s}  \lVert F \rVert_{\calL^{2}(W, \sigma^{*}(\epsilon^{-2}\bar{g}))} \\
& < \frac{8n}{n - 1} \bar{D}_{0} D_{5} \epsilon^{2} \eta + D_{0}D_{2}D_{s} \epsilon^{2} \text{Vol}_{g}(X) (C + 1) \epsilon^{\frac{1}{2}} \\
& : = \bar{D} \epsilon^{2} \eta + \bar{D}' \epsilon^{2 + \frac{1}{2}} (C + 1).
\end{split}
\end{equation}
It follows by the smallness of $ \eta $ in Proposition \ref{Set:prop1} that
\begin{equation*}
\lVert \frac{\partial^{2}u}{\partial t^{2}} \rVert_{\calC^{0}(U_{1})} = \epsilon^{-2} \lVert \frac{\partial^{2}u}{\partial (t')^{2}} \rVert_{\calC^{0}(O_{1})} < \bar{D} \eta + \bar{D}'(C + 1) \epsilon^{\frac{1}{2}} : = \eta' \ll 1.
\end{equation*}
As desired.
\end{proof}
\begin{remark}\label{Set:re0a}
Lemma \ref{Set:lemma4} also holds if we replace the bounded geometry assumption by bounded curvature, since the comparison of volume forms still hold by \cite[Lemma 2.26]{Aubin}.
\end{remark}
\begin{remark}\label{Set:re1}
    Since $ u_{0} = u + 1 $ and $ \Phi \equiv 1 $ in $ X \times [-1, 1]_{\xi} \times \mathbb{S}_{t}^{1} $, Lemma \ref{Set:lemma4} implies that
    \begin{equation*}
        \left\lVert \frac{\partial^{2} u_{0}}{\partial t^{2}} \right\rVert_{\calC^{0}(X \times \lbrace 0 \rbrace_{\xi} \times \lbrace P \rbrace_{t})} < \eta' \Rightarrow \left\lvert \frac{\partial^{2} u'}{\partial t^{2}} \bigg|_{X \times [-1, 1]_{\xi} \times \lbrace P \rbrace_{t}} \right\rvert < \eta'.
    \end{equation*}
\end{remark}
As mentioned above, we need to compare the Laplacians on $(M, g)$,  $(W, \sigma^*\bar{g})$, and  $ (M \times \mathbb{S}^{1},\bar{g}) $ with respect to functions $ \tilde{u}' $ on $ M $, $ u_{0} $ on $ W $ and $ u' $ on $ M \times \mathbb{S}^{1} $, respectively. It is essentially due to \cite[Lemma 2.4]{RX2}. For completeness, we also give a detailed proof here.
\begin{lemma}\label{Set:lemma5}
Under the hypotheses of Lemma \ref{Set:lemma4},
\begin{equation}\label{Set:eqn12}
\begin{split}
\Delta_{\bar{g}} u' & = \Delta_{\tau_{1}^{*}\bar{g}} \tilde{u}' + \frac{\partial^{2} u'}{\partial t^{2}} = \Delta_{g} \tilde{u}' + \frac{\partial^{2} u'}{\partial t^{2}} \; {\rm on} \; X \times \lbrace 0 \rbrace_{\xi} \times \lbrace P \rbrace_{t}, \\
\Delta_{\bar{g}}u' & = \Delta_{\sigma^{*}\bar{g}} u_{0} + A_{1}(\bar{g}, M, u) \; {\rm on} \; X \times \lbrace 0 \rbrace_{\xi} \times \lbrace P \rbrace_{t},
\end{split}
\end{equation}
where $ A_{1}(\bar{g}, M, u) $ involves zeroth and first order derivatives of $ u $, which does not contain $ \partial_{\xi} $- terms.
\end{lemma}
\begin{proof}
Since Laplacians are invariant under different choices of coordinates, we analyze (\ref{Set:eqn12}) locally. For $ \pi:M \times \mathbb{S}^{1} \to W$, we are computing
\begin{equation}\label{Set:eqn13}
\Delta_{\bar{g}} u'  - \pi^{*}(\Delta_{\sigma^{*}\bar{g}}u_{0}) = \Delta_{\bar{g}} (u_{0} \circ \pi) - \pi^{*}(\Delta_{\sigma^{*}\bar{g}}u_{0}) = \Delta_{\bar{g}} (u_{0} \circ \pi) - \pi^{*}(\Delta_{\sigma^{*}\bar{g}}(u' \circ \sigma))
\end{equation}
in a tubular neighborhood $ T$ of the hypersurface $ W = W \times \lbrace 0 \rbrace_{\xi} \subset M$, and then restricting to $ W $. According to the exponential map with respect to the normal vector field along $ W $, we can take a chart $(x^1,\ldots,x^n, t)$ in some open subset $ B \subset T \subset M $ containing a chart $(x^1,\ldots,x^{n - 1}, t)$ for  
$B \cap W\times \lbrace 0 \rbrace_{\xi} $, with $ t : = x^{n + 1} $  the coordinate on $ [-1, 1] \subset \mathbb{S}_{t}^{1}
$ and $ \partial_{ x^{n}}$ normal to the level surfaces.
In addition, $ B \cap W $ is characterized by $ \lbrace x^n=0 \rbrace $ in  these coordinates. It follows that we can write $ g = g_{X, x_{n}} \oplus d(x^{n})^2 $ 
and $ \bar{g} = g_{X, x_{n}} \oplus d(x^{n})^2 \oplus dt^{2} $ near  
$ W \times \lbrace 0 \rbrace_{\xi} $. 
Therefore, $ \sigma^{*}\bar{g} = g_{X, 0} \oplus dt^{2} $ on $ W \approx W \times \lbrace 0 \rbrace_{\xi} $, 
and $ \bar{g}_{ij} = (\sigma^{*}\bar{g})_{ij}, i, j = 1, \dotso, n - 1 $ on $ W $.
In summary, on $ B \cap T $,
\begin{equation*}
    (\bar{g}_{ij}) = \begin{pmatrix} \bar{g}_{11} & \dotso & \bar{g}_{1, n - 1} & 0 & 0 \\
    \vdots & \dotso & \vdots & \vdots & \vdots \\
    \bar{g}_{n -1, 1} & \dotso & \bar{g}_{n -1, n - 1} & 0 & 0 \\
    0 & \dotso & 0 & 1 & 0 \\
    0 & \dotso & 0 & 0 & 1 \end{pmatrix},
\end{equation*}
with $ (\bar{g}^{ij}) $ the inverse of the whole matrix, and $ (\sigma^{*}\bar{g})^{ij} $  the inverse of the minor given by deleting the $n^{\rm th}$ row and column. With this choice of local coordinates, we have $ \bar{g}^{ij} = (\sigma^{*}\bar{g})^{ij}, i, j = 1, \dotso, n - 1 $ on $ B \cap W $.

Then
\begin{equation*}
\Delta_{g} u' = \sum_{i, j= 1}^{n- 1} \bar{g}^{ij} \frac{\partial^{2} u'}{\partial x^{i} \partial x^{j}} + \frac{\partial^{2} u'}{\partial (x^{n})^{2}} + \frac{\partial^{2}u'}{\partial t^{2}} - \sum_{i, j, k = 1}^{n + 1} \bar{g}^{ij} \Gamma_{ij, \bar{g}}^{k} \frac{\partial u'}{\partial x^{k}}.
\end{equation*}
The projection map $ \pi : M \times \mathbb{S}^{1} \rightarrow W $ 
is $ \pi (x^{1}, \dotso, x^{n -1}, \xi, t) = (x^{1}, \dotso, x^{n-1}, t) $, 
so by chain rule
\begin{equation*}
    \frac{\partial u'}{\partial x^{n}} = \frac{\partial (u_{0} \circ \pi)}{\partial x^{n}} =0 \; {\rm on} \; B \cap T.
\end{equation*}
The inclusion map $ \sigma : W \rightarrow M \times \mathbb{S}^{1} $ is $ \sigma(x^{1}, \dotso, x^{n-1}, t) = (x^{1}, \dotso, x^{n -1}, \xi, t) $. Similarly, it is easy to check by chain rule and $ \frac{\partial u'}{\partial \xi} = 0 $ that 
\begin{equation*}
\frac{\partial u_{0}}{\partial x^{i}} \biggl|_{B \cap W} = \frac{\partial (u' \circ \sigma)}{\partial x^{i}} \biggl|_{B \cap W} = 
  \frac{\partial u'}{\partial x^{i}}\biggl|_{B \cap W}, i = 1, \dotso, n - 1, n + 1.
\end{equation*}
Since $ (\bar{g}^{ij}) = (\sigma^{*}\bar{g})^{ij} $ on $ B \cap W $, (\ref{Set:eqn13}) becomes
we have
\begin{align*}
   & 
   \Delta_{\bar{g}}u' |_{B \cap W} - \Delta_{\sigma^{*}\bar{g}}u_{0} \\
   & \qquad = \left( \sum_{i, j= 1}^{n- 1} \bar{g}^{ij} \frac{\partial^{2} u'
   }{\partial x^{i} \partial x^{j}}\biggl|_{B \cap W} + \frac{\partial^{2}u'
   }{\partial t^{2}}\biggl|_{B \cap W} - \sum_{i, j,k = 1}^{n + 1} \bar{g}^{ij} \Gamma_{ij, \bar{g}}^{k} \frac{\partial u'
   }{\partial x^{k}}\biggl|_{B \cap W} \right) \\
   & \qquad\qquad - \left( \sum_{i, j
   = 1}^{n-1} (\sigma^{*}\bar{g})^{ij} \frac{\partial^{2}u'
   }{\partial x^{i} \partial x^{j}}\biggl|_{B \cap W} + \frac{\partial^{2}u'
   }{\partial t^{2}}\biggl|_{B \cap W} - \sum_{i, j, k \neq n}(\sigma^{*}\bar{g})^{ij} \Gamma_{ij, \sigma^{*}\bar{g}}^{k} \frac{\partial u'
   }{\partial x^{k}}\biggl|_{B \cap W} \right) \\
   & \qquad : = A_{1}(\bar{g}, M, u) |_{B \cap W}.
\end{align*}
The second equality of (\ref{Set:eqn12}) thus follows.

The first equality of (\ref{Set:eqn12}) follows from the fact that $ \bar{g} = \tau_{1}^{*} \bar{g} \oplus dt^{2} = g \oplus dt^{2} $ with $ t = x^{n + 1} $ with the same local coordinate chart, and hence all Christoffel symbols involving the $ x^{n + 1} $ direction vanish.
\end{proof}
\medskip
\begin{remark}\label{Set:re2}
For fixed $\eta$ in Prop.~\ref{Set:prop1}, for  
fixed $P \in \mathbb{S}^{1} $ and $ u_{0} = u + 1 $ as in (\ref{Set:eqn7}), there exists a constant $ K_1 = K_1(\eta) > 0 $
such that
\begin{equation*}
  4 |A_1(\bar{g}, M, u)| <K_{1} \; {\rm on} ; X \times \lbrace 0 \rbrace_{\xi} \times \lbrace P \rbrace_{t}.  
\end{equation*}  
We have $K_1\to 0$ as $\eta\to 0 $. It is straightforward since $ A_{1} $ only consists of zeroth and first order derivatives of $ u $.
\end{remark}

To close this section, we need the following conformal relations.
On appropriate spaces, we define
\begin{equation*}
    e^{2\phi'} : = (u')^{\frac{4}{n - 2}} \Rightarrow \phi_{0} : = \sigma^{*} \phi', \tilde{\phi}' : = \tau_{1}^{*} \phi' \Rightarrow e^{2\phi_{0}} : = u_{0}^{\frac{4}{n - 2}}, e^{\tilde{\phi}'} = (\tilde{u}')^{\frac{4}{n - 2}}.
\end{equation*}
For any two functions $ \varphi, v $ with the relation $ e^{2\varphi} = v^{\frac{4}{n - 2}}, n \geqslant 3 $ and any Riemannian metric $ g_{0} $, we have
\begin{equation}\label{Set:eqn14}
\begin{split}
    & v^{-\frac{n+2}{n- 2}}\left(-\frac{4}{n - 2} \Delta_{g_{0}} v 
    \right) = e^{-2\varphi} \left(  
    - 2 \Delta_{g_{0}} \varphi - (n - 2) \lvert \nabla_{g_{0}} \varphi \rvert^{2} \right), \\
   &  e^{-2\varphi} \nabla_{g_{0}} \varphi = \frac{2}{n - 2} v^{-\frac{n+2}{n-2}} \nabla_{g_{0}} v,\  e^{-2\varphi} \lvert \nabla_{g_{0}} \varphi \rvert^{2} = \left( \frac{2}{n - 2} \right)^{2}v^{-\frac{2n}{n - 2}} \lvert \nabla_{g_{0}} v \rvert^{2}, \\
   & e^{-2\tilde{\phi}'} \left( 2(n - 2) \nabla_{V} \nabla_{V} \tilde{\phi}' + (n - 2)^{2} \nabla_{V} \tilde{\phi}' \nabla_{V} \tilde{\phi}' \right)   
= 4(\tilde{u}')^{-\frac{n + 2}{n - 2}}   \nabla_{V} \nabla_{V} \tilde{u}'.
\end{split}
\end{equation}  
\medskip

\section{Existence of Positive Scalar Curvature Metric on $ X $}
In this section, we show that for any complete metric $ g $ on the noncompact cylinder $ M $ such that (i) $ g $ is of bounded geometry, (ii) $ R_{g} \geqslant \kappa_{0} > 0 $  for some $ \kappa_{0} \in R $ and (iii) satisfying the $ g $-angle condition $ \angle_{g}(\nu_{g}, \partial_{\xi}) < \frac{\pi}{4} $ on $ X_{0} $, there exists a conformal metric $ \tilde{g}' \in [g] $ such that $ R_{\imath^{*} \tilde{g}'} > 0 $ on $ X $. 

We point out that once we introduce the notions of positivity that are compatible with noncompact manifolds, and apply the positivity to give $ \calC^{1, \alpha} $- and partial $ \calC^{2} $-estimates, the following argument is very similar as in \cite[Theorem 3.1]{RX2} and \cite[Theorem 3.1]{XU10}. Clearly, we impose some geometric assumptions when dealing with noncompact manifolds.
\begin{theorem}\label{PSC:thm1}
Let $ X $ be an oriented, closed manifold. Let $ (M = X \times \R, g) $ be a Riemannian noncompact cylinder of bounded geometry, $ n = \dim M \geqslant 3 $. If $ R_{g} \geqslant \kappa_{0} > 0 $ on $ M $ for some $ \kappa_{0} \in \R $, and
\begin{equation*}
        \angle_{g}(\nu_{g}, \partial_{\xi}) < \frac{\pi}{4} \; {\rm on} \; X_{0},
\end{equation*}
then there exists a metric $ \tilde{g} $ in the conformal class of $ g $ such that $ \imath^{*} \tilde{g} $ is a PSC metric on $ X \cong X_{0} $.
\end{theorem}
\begin{proof}
We set $ p, \alpha $ as in Proposition \ref{Set:prop1}. Choose $ C > 0 $ such that
\begin{equation}\label{POS:eqn1}
C > 2 \max_{X \times \lbrace 0 \rbrace_{\xi} \times \lbrace P \rbrace_{t}} \left( \lvert R_{g} \rvert + 2 \lvert {\rm Ric}_{g}(\nu_{g}, \nu_{g}) \rvert + h_{g}^{2} + \lvert A_{g} \rvert^{2} \right)  + 3.
\end{equation}
Note that $ \tau_{1}^{*}\bar{g} = g $, $ \nabla_{\tau_{1}^{*} \bar{g}} $ and $ \nabla_{V} $ do not contain $ \xi $-derivative. We choose $ \eta, \eta' \ll 1 $, and associated small enough $ \delta $, $ \epsilon $, and finally associated $ F $ in Lemma \ref{Set:lemma2}, such that (i) the solution of (\ref{Set:eqn3}) satisfies  (\ref{Set:eqn4}) and (\ref{Set:eqn8}); (ii) $ \lvert A_{1}(\bar{g}, M, u) \rvert < 1 $ on $ X \times \lbrace 0 \rbrace_{\xi} \times \lbrace P \rbrace_{t} $; and (iii)
\begin{equation}\label{POS:eqn1a}
   \frac{1}{2} < \lVert u_{0} \rVert_{\calC^{0}(X \times [-\frac{\epsilon}{2}, \frac{\epsilon}{2}]_{t})} < \frac{3}{2} \Rightarrow \frac{4}{n - 2} \left\lvert  \frac{\left\lvert \nabla_{\tau_{1}^{*} \bar{g}} u_{0} \right\rvert^{2}}{u_{0}} + n \frac{\nabla_{V} u_{0} \nabla_{V} u_{0}}{u_{0}} \right\rvert < 1 \; {\rm on} \; X \cong X \times \lbrace 0 \rbrace_{\xi}.
\end{equation}
For the metric $ \tilde{g} = e^{2\phi'} \bar{g} = (u')^{\frac{4}{n - 2}} \bar{g} $ on $ M \times \mathbb{S}^{1}_{t} $, the normal vector field along $ X \times \lbrace 0 \rbrace_{\xi} \times \lbrace P \rbrace_{t} $ becomes $ e^{-\tilde{\phi}'}\nu_{g} $. Applying the Gauss-Codazzi equation on $ X \cong X \times \lbrace 0 \rbrace_{\xi} \times \lbrace P \rbrace_{t} $ with respect to the metric $ \tau_{1}^{*}\tilde{g} $ on $ M \cong X \times \R_{\xi} \times \lbrace P \rbrace_{t} $,
\begin{equation*}
R_{\imath^{*} \tau_{1}^{*} \tilde{g}}
 = R_{\tau_{1}^{*} \tilde{g}} - 2{\rm Ric}_{\tau_{1}^{*} \tilde{g}}\left(e^{-\tilde{\phi}'}\nu_{g}, e^{-\tilde{\phi}'}\nu_{g} \right) + h_{\tau_{1}^{*} \tilde{g}} - \lvert A_{\tau_{1}^{*} \tilde{g}} \rvert^{2}.
\end{equation*}
For $ \tau_{1}^{*} \tilde{g} = \tau_{1}^{*} \left( e^{2\phi'} \bar{g} \right) = e^{2\tilde{\phi}'} \tau_{1}^{*} \bar{g}$, the conformal transformation of Ricci and scalar curvatures are given by
\begin{align*}
    {\rm Ric}_{\tau_{1}^{*} \tilde{g}}\left(e^{-\tilde{\phi}'}\nu_{g}, e^{-\tilde{\phi}'}\nu_{g}\right) & = e^{-2\tilde{\phi}'}\left({\rm Ric}_{\tau_{1}^{*}\bar{g}}\left(\nu_{g}, \nu_{g} \right) - (n - 2)(\nabla_{\nu_{g}} \nabla_{\nu_{g}} \tilde{\phi}' - \nabla_{\nu_{g}} \tilde{\phi}' \nabla_{\nu_{g}} \tilde{\phi}')\right) \\
    & \qquad - e^{-2\tilde{\phi}'}\left(\Delta_{\tau_{1}^{*}\bar{g}} \tilde{\phi}' + (n - 2) \left\lvert \nabla_{\tau_{1}^{*} \bar{g}} \tilde{\phi}' \right\rvert^{2}\right)\bar{g}_{\xi\xi},\\
  R_{\tau_{1}^{*} \tilde{g}}  &=  e^{-2\tilde{\phi}'} \left(R_{\tau_{1}^{*}\bar{g}} -2(n-1)\Delta_{\tau_{1}^{*}\bar{g}}\tilde{\phi}'
  -(n-2)(n-1)\left\lvert\nabla_{\tau_{1}^{*}\bar{g}}\tilde{\phi}'\right\rvert^2 \right);
  \end{align*}
in addition, the conformal transformation of the second fundamental form and mean curvature are given by
\begin{align*}
   \lvert A_{\tau_{1}^{*} \tilde{g}} \rvert^{2} & = e^{-2\tilde{\phi}'} \left( \lvert A_{\tau_{1}^{*} \bar{g}} \rvert^{2} + 
   2 n h_{\tau_{1}^{*}\bar{g}} \frac{\partial \tilde{\phi}'}{\partial \nu_{g}} +  n^2
   \left(\frac{\partial \tilde{\phi}'}{\partial \nu_{g}} \right)^{2}  \right),\\
    h_{\tau_{1}^{*} \tilde{g}}^{2} & = e^{-2\tilde{\phi}'} \left( h_{\tau_{1}^{*} \bar{g}}^{2} +
     2n h_{\tau_{1}^{*}\bar{g}} \frac{\partial \tilde{\phi}'}{\partial \nu_{g}}  + 
       n^2\left(\frac{\partial \tilde{\phi}'}{\partial \nu_{g}} \right)^{2} \right).
\end{align*}
By (\ref{Set:eqn6}), $ \nabla_{\nu_{g}} \nabla_{\nu_{g}} \tilde{\phi}' = \nabla_{V} \nabla_{V} \tilde{\phi}' $ in a tubular neighborhood of $ X \times \lbrace 0 \rbrace_{\xi} $. According to the laws of conformal transformations listed above and the differential relations in (\ref{Set:eqn14}), the Gauss-Codazzi equation converts to
\begin{equation}\label{POS:eqn2}
    \begin{split}
    R_{\imath^{*} \tau_{1}^{*} \tilde{g}} & = e^{-2\tilde{\phi}'} \left(R_{\tau_{1}^{*}\bar{g}} - 2{\rm Ric}_{\tau_{1}^{*}\bar{g}}\left(\nu_{g}, \nu_{g} \right) + h_{\tau_{1}^{*} \bar{g}}^{2} - \lvert A_{\tau_{1}^{*} \bar{g}} \rvert^{2} \right) \\
    & \qquad + e^{-2\tilde{\phi}'} \left(- 2(n - 2) \Delta_{\tau_{1}^{*}\bar{g}} \tilde{\phi}' - (n - 3)(n - 2) \left\lvert \nabla_{\tau_{1}^{*}\bar{g}} \tilde{\phi}' \right\rvert \right) \\
    & \qquad \qquad + 2(n - 2)e^{-2\tilde{\phi}'}\left(\nabla_{V} \nabla_{V} \tilde{\phi}' -  \nabla_{V} \tilde{\phi}' \nabla_{V} \tilde{\phi}' \right) \\
    & = (\tilde{u}')^{-\frac{n +2}{n - 2}}\left(R_{\tau_{1}^{*}\bar{g}} \tilde{u}' - 2{\rm Ric}_{\tau_{1}^{*}\bar{g}}\left(\nu_{g}, \nu_{g} \right)\tilde{u}' + h_{\tau_{1}^{*} \bar{g}}^{2}\tilde{u}' - \lvert A_{\tau_{1}^{*} \bar{g}} \rvert^{2}\tilde{u}' \right) \\
    & \qquad + (\tilde{u}')^{-\frac{n +2}{n - 2}}\left( 4\nabla_{V} \nabla_{V} \tilde{u}' - 4\Delta_{\tau_{1}^{*}\bar{g}} \tilde{u}' \right) \\
    & \qquad \qquad + \frac{4}{n - 2}(\tilde{u}')^{-\frac{n +2}{n - 2}} \left(  \frac{\left\lvert \nabla_{\tau_{1}^{*} \bar{g}} \tilde{u}' \right\rvert^{2}}{\tilde{u}'} - n \frac{\nabla_{V} \tilde{u}' \nabla_{V} \tilde{u}'}{\tilde{u}'} \right).
    \end{split}
\end{equation}
Note that on $ X \cong X \times \lbrace 0 \rbrace_{\xi} \times \lbrace P \rbrace_{t} \subset W \times \lbrace 0 \rbrace_{\xi} $, $ \tilde{u}' = u_{0} $. Recall that $ u_{0} = u + 1 $ where $ u $ is the solution of (\ref{Set:eqn3}), it follows that $ u_{0} $ satisfies the following PDE
\begin{equation*}
    4\nabla_{V} \nabla_{V} u_{0} - 4\Delta_{\sigma^{*} \bar{g}} u_{0} + R_{\bar{g}} |_{W} u_{0} = F + R_{\bar{g}} |_{W} \; {\rm on} \; X.
\end{equation*}
By $ \bar{g} = g \oplus dt^{2} $, $ R_{\bar{g}} = R_{\tau_{1}^{*}\bar{g}} = R_{g} > 0 $ on $ X \times \lbrace 0 \rbrace_{\xi} \times \lbrace P \rbrace_{t} $ by hypothesis. By (\ref{Set:eqn8}) of Lemma \ref{Set:lemma4}, (\ref{Set:eqn12}) of Lemma \ref{Set:lemma5}, Remark \ref{Set:re1}, and (\ref{POS:eqn1a}), the formula (\ref{POS:eqn2}) on $ X \cong X \times \lbrace 0 \rbrace_{\xi} \times \lbrace P \rbrace_{t} $ becomes
\begin{equation}\label{POS:eqn3}
\begin{split}
 R_{\imath^{*} \tau_{1}^{*} \tilde{g}} & =u_{0}^{-\frac{n +2}{n - 2}}\left(  - 2{\rm Ric}_{\tau_{1}^{*}\bar{g}}\left(\nu_{g}, \nu_{g} \right)u_{0} + h_{\tau_{1}^{*} \bar{g}}^{2} u_{0} - \lvert A_{\tau_{1}^{*} \bar{g}} \rvert^{2} u_{0} \right) \\
    & \qquad + u_{0}^{-\frac{n +2}{n - 2}}\left( 4\nabla_{V} \nabla_{V} u_{0} - 4\Delta_{\tau_{1}^{*}\bar{g}} \tilde{u}' - 4\frac{\partial^{2} u'}{\partial t^{2}} + 4\frac{\partial^{2} u'}{\partial t^{2}} + R_{\tau_{1}^{*} \bar{g}} u_{0} \right) \\
    & \qquad \qquad + \frac{4}{n - 2} u_{0}^{-\frac{n +2}{n - 2}} \left( \frac{\left\lvert \nabla_{\tau_{1}^{*} \bar{g}} u_{0} \right\rvert^{2}}{u_{0}} - n \frac{\nabla_{V} u_{0} \nabla_{V} u_{0}}{u_{0}} \right) \\
    & \geqslant u_{0}^{-\frac{n +2}{n - 2}}\left(- 2{\rm Ric}_{\tau_{1}^{*}\bar{g}}\left(\nu_{g}, \nu_{g} \right)u_{0} + h_{\tau_{1}^{*} \bar{g}}^{2} u_{0} - \lvert A_{\tau_{1}^{*} \bar{g}} \rvert^{2} u_{0} \right) \\
    & \qquad + u_{0}^{-\frac{n +2}{n - 2}}\left( 4\nabla_{V} \nabla_{V} u_{0} - 4\Delta_{\bar{g}} u' - 4\eta' + R_{\bar{g}} |_{W} u_{0}\right) + u_{0}^{-\frac{n +2}{n - 2}} \cdot (-1) \\
    & = u_{0}^{-\frac{n +2}{n - 2}}\left( - 2{\rm Ric}_{\tau_{1}^{*}\bar{g}}\left(\nu_{g}, \nu_{g} \right)u_{0} + h_{\tau_{1}^{*} \bar{g}}^{2} u_{0} - \lvert A_{\tau_{1}^{*} \bar{g}} \rvert^{2} u_{0} \right) \\
    & \qquad + u_{0}^{-\frac{n +2}{n - 2}}\left( 4\nabla_{V} \nabla_{V} u_{0} - 4\Delta_{\sigma^{*}\bar{g}} u_{0}  + R_{\bar{g}} |_{W} u_{0}\right) + u_{0}^{-\frac{n +2}{n - 2}} \cdot (-1 - 4 \eta' + A_{1}(\bar{g}, M, u)) \\
    & = u_{0}^{-\frac{n +2}{n - 2}}\left(- 2{\rm Ric}_{\tau_{1}^{*}\bar{g}}\left(\nu_{g}, \nu_{g} \right)u_{0} + h_{\tau_{1}^{*} \bar{g}}^{2} u_{0} - \lvert A_{\tau_{1}^{*} \bar{g}} \rvert^{2} u_{0} \right) \\
    & \qquad u_{0}^{-\frac{n +2}{n - 2}} \left(F + R_{\bar{g}} |_{W} -1 - 4 \eta' + A_{1}(\bar{g}, M, u) \right).
\end{split}
\end{equation}
By Lemma \ref{Set:lemma2}, $ F = C + 1 $ on $ \partial M $. By (\ref{POS:eqn1}) and (\ref{POS:eqn3}), it follows that
\begin{equation*}
    R_{\imath^{*} \tau_{1}^{*} \tilde{g}} \geqslant  u_{0}^{-\frac{n +2}{n - 2}} \left( C + 1 + R_{\bar{g}} - 2 \max_{\partial M} \left(2 \lvert {\rm Ric}_{g}(\nu_{g}, \nu_{g}) \rvert + h_{g}^{2} + \lvert A_{g} \rvert^{2} \right) - 3 \right) > 0. 
\end{equation*}
Therefore, the metric
\begin{equation*}
 \tau_{1}^{*} \tilde{g} = \tau_{1}^{*}(u')^{\frac{4}{n- 2}} \bar{g} = (\tilde{u}')^{\frac{4}{n- 2}} g : = g', g' \in [g]
\end{equation*}
on $ M $ induces a PSC metric $ \imath^{*} g' \in [\imath^{*} g] $ on $ X \cong X \times \lbrace 0 \rbrace_{\xi} $.
\end{proof}
\begin{remark}\label{PSC:re1}
The conclusion of Theorem \ref{PSC:thm1} still holds if we replace the bounded geometry assumption by bounded curvature assumption due to Remark \ref{Set:re0} and Remark \ref{Set:re0a}.
\end{remark}
With Lemma \ref{Set:lemma3} and the Yamabe problem on closed manifolds \cite{PL}, an immediate consequence follows in terms of the transpose of the positivity of the Yamabe constants from $ X \times \R $ to $ X $:
\begin{corollary}\label{POS:cor1}
Let $ X $ be an oriented, closed manifold. Let $ (M = X \times \R, g) $ be a Riemannian noncompact cylinder of bounded geometry or bounded curvature such that $ n : = \dim M \geqslant 3 $. If $ R_{g} \geqslant \kappa_{0} > 0 $, then the Yamabe constant of the conformal class $ [g] $ on $ M $ is positive. If in addition,
    \begin{equation*}
        \angle_{g}(\nu_{g}, \partial_{\xi}) \in [0, \frac{\pi}{4})
    \end{equation*}
along the hypersurface $ X_{0} $, then the Yamabe constant of the conformal class $ [\imath^{*}g] $ on $ X $ is also positive.
\end{corollary}
\medskip

\section{Positive Scalar Curvature on Compact Manifolds with Boundary}
Throughout this section, we denote our oriented, compact manifold with nonempty smooth boundary by $ (X, \partial X) $, $ \dim X \geqslant 3 $. In previous sections, we analyze how to transpose the positivity of scalar curvature, and the positivity of the Yamabe constant, from the noncompact cylinder $ X \times \R $ to oriented closed manifolds $ X $, i.e. $ \partial X = \emptyset $. In this section, we introduce a notion of Yamabe constant on noncompact cylinders with noncompact boundary of cylindral type. We then show analogously that if $ (M: = X \times \R_{\xi}, \partial M : = \partial X \times \R_{\xi}) $, a noncompact cylinder with noncompact boundary of cylindrical type, which admits a metric $ g $ with positive Yamabe constant, then $ (X, \partial X) $ admits a metric with positive Yamabe constant, provided that an analogous $ g $-angle condition and some geometric hypotheses on $ g $ are imposed. 

Due to the introduction of the non-empty boundaries $ \partial M $ and $ \partial X $, we say that a Riemannian metric $ g $ on $ (M, \partial M) $ is of {\it{bounded geometry}} if it is in the sense of Schick \cite[Definition 2.2]{Schick}; we have to define a new notion of positivity on $ (M, \partial M) $ and a $ g $-angle condition along the hypersurface $ (X \times \lbrace 0 \rbrace_{\xi}, \partial X \times \lbrace 0 \rbrace_{\xi}) $ of $ (M, \partial M) $; we also need to construct a different elliptic partial differential equation analogous to (\ref{Set:eqn3}) by introducing some boundary condition; consequentially, we need to modify our partial $ \calC^{2} $-estimate argument in (\ref{Set:eqn8}), especially near the boundary. With these preparations, all other arguments are almost the same as we gave in \S2 and \S3, and hence the conclusion follows. We again pair the space $ M $ with $ \mathbb{S}_{t}^{1} $, as shown below by a slightly modified diagram:
\begin{equation}\label{diagram1}
\begin{tikzcd} ((M, \partial M), g) \arrow[r, "\tau_{1}"] & ((M \times \mathbb{S}^{1}_{t}, \partial M \times \mathbb{S}_{t}^{1}), g \oplus dt^{2} = \bar{g}) \arrow[d, shift left = 1.5ex, "\pi"] \\
((X, \partial X), \imath^*g) \arrow[r,"\tau_{2}"]\arrow[u,"\imath"] &  ((W = X \times \mathbb{S}^{1}_{t}, \partial W = \partial X \times \mathbb{S}_{t}^{2}), \sigma^*(\bar{g})) 
\arrow[u,"\sigma"]
\end{tikzcd}
\end{equation}
We always identify $ (X, \partial X) $ with $ (X \times \lbrace 0 \rbrace_{\xi}, \partial X \times \lbrace 0 \rbrace_{\xi}) $ or $ (X \times \lbrace 0 \rbrace_{t}, \partial X \times \lbrace 0 \rbrace_{t}) $. Similarly, $ (W, \partial W) \cong (W \times \lbrace 0 \rbrace_{\xi}, \partial W \times \lbrace 0 \rbrace_{\xi}) $. Again we denote the Riemannian metric on $ M \times \mathbb{S}^{1} $ by $ \bar{g} : = g \oplus dt^{2} $, smooth up to the boundary. We define $ \sigma : (W, \partial W) \rightarrow (M \times \mathbb{S}^{1}_{t}, \partial M \times \mathbb{S}_{t}^{1}) $ is given by $ \sigma(w) = (w, 0) $ with fixed point $ 0 \in \R_{\xi} $. In particular, the restriction of $ \sigma $ on $ \partial W $ satisfies $ \sigma(\partial W) \subset \partial M \times \mathbb{S}_{t}^{1} $, i.e. if $ w \in \partial W $, then $ \sigma(w) = (w, 0) \in \partial M \times \mathbb{S}_{t}^{1} $. Fix a point $ P \in \mathbb{S}^{1} $, we also define $ \tau_{1} : M \rightarrow M \times \mathbb{S}^{1}, \tau_{2} : X \rightarrow W $ by sending $ M \ni x \mapsto (x, P) = \tau_{1}(x), X \ni y \mapsto (y, P) = \tau_{2}(y) $, respectively. Similarly, we still denoted by $ \tau_{1}, \tau_{2} $ their restrictions on $ \partial M, \partial X $, respectively. Under local parametrization, we may identify $ \lbrace P \rbrace_{t} $ with point $ \lbrace 0 \rbrace_{t} $ in any local chart containing $ P $. We use the labels $ \lbrace 0 \rbrace_{\xi} \in \R_{\xi} $ and $ \lbrace 0 \rbrace_{t} \cong \lbrace P \rbrace \in \mathbb{S}_{t}^{1} $ to distinguish these two points.

Set $ X_{0} : = (X \times \lbrace 0 \rbrace_{\xi}, \partial X \times \lbrace 0 \rbrace_{\xi}) \subset (M, \partial M) $. The unit normal vector field $ \nu_{g} $ along $ X_{0} $ is still well-defined, up to the boundary. Therefore, the $ g $-angle quantity we now introduce on $ X_{0} $ is exactly the same as (\ref{Intro:eqn1}). It follows that we can also define the smooth global vector field
\begin{equation*}
V = \nu_{g} - g(\nu_{g}, \partial_{\xi})^{-1} \cdot \partial_{\xi} \in \Gamma(T X_{0})
\end{equation*}
up to the boundary, which extends to a smooth global vector field on $ X_{0} \times \mathbb{S}_{t}^{1} \cong W $, up to the boundary. We still denote such a vector field by $ V $. If we assume
\begin{equation}\label{BD:eqn1}
\angle_{g}(\nu_{g}, \partial_{\xi}) \in [0, \frac{\pi}{4}) \; {\rm on} \; X_{0},
\end{equation}
then the operator
\begin{equation*}
\nabla_{V} \nabla_{V} - \Delta_{\sigma^{*} \bar{g}}
\end{equation*}
is again an elliptic operator on $ (W, \partial W, \sigma^{*} \bar{g}) $. 

By Escobar \cite{ESC}, we can apply a conformal transformation $\imath^{*}g \mapsto e^{2f} \imath^{*}g $ on $ (X, \partial X, \imath^{*}g) $ by some smooth function $ f : X \rightarrow \R $, such that $ h_{e^{2f}\imath^{*}g} = 0 $ on $ \partial X $. Pulling back $ f \mapsto \imath^{*} f : M \rightarrow \R $, it follows that the conformal metric $ e^{2\pi^{*}f} g $ has the property that the induced metric $ \imath^{*} (e^{2\pi^{*}f} g) $ has zero mean curvature on $ \partial X $. Since the angle is preserved under conformal transformation, we may assume that the initial metric $ g $ satisfies $ h_{\imath^{*}g} = 0 $ on $ \partial X $ from now on.

The inhomogeneous term $ F $ is constructed exactly the same as in Lemma \ref{Set:lemma2}. The only difference is that the domain $ W $ has nonempty smooth boundary, where the smoothness of $ F $ extends to. Let $ \nu $ be the unit normal vector field along $ \partial X $ with respect to the induced metric $ \imath^{*} g $. Since $ \bar{g} = g \oplus dt^{2} $, the same vector field extends to the unit normal vector field along $ \partial W $ in terms of the pushforward map of the inclusion $ \tau_{2} $. We still denote it by $ \nu $. We are now ready to introduce the partial differential equation on $ (W, \partial W, \sigma^{*} \bar{g}) $ analogous to Proposition \ref{Set:prop1}.
\begin{proposition}\label{BD:prop1}
Let $ (M, \partial M, g) $ be a noncompact cylinder of bounded geometry. Let $ (W, \partial W, \sigma^{*} \bar{g}) $ be as above. Assume that (\ref{BD:eqn1}) holds. Fix a constant $ R > 0 $. For any positive constant $ \kappa \ll 1 $, any positive constant $ C $, and any $ p > n = \dim W $, there exists an associated $ F $ and $ \delta $ in the sense of Lemma \ref{Set:lemma2}, such that the following partial differential equation     
\begin{equation}\label{BD:eqn2}
L u : = 4\nabla_{V} \nabla_{V} u - 4\Delta_{\sigma^{*} \bar{g}}  u + Ru = F  \; {\rm in} \; W; \frac{\partial u}{\partial \nu} = 0 \; {\rm on} \; \partial W.
\end{equation}
admits a unique smooth solution $ u $ with
\begin{equation}\label{BD:eqn3}
\lVert u \rVert_{\calC^{1, \alpha}(W)} < \kappa
\end{equation}
for some $ \alpha \in (0, 1) $ such that $ \alpha \geqslant 1 - \frac{n}{p} $. 
\end{proposition}
\begin{proof}
Fix $ \kappa \ll 1, C $, and $ p > n $. We then fix some $ \alpha \in (0, 1) $ such that $ 1 + \alpha \geqslant 2 - \frac{n}{p} $. Denote $ C = C(W, g, n, p, L) $ by the constant of $ \calL^{p} $ elliptic regularity estimates with respect to $ L $, and $ C' = C'(W, g, n, p, \alpha) $ by the constant of the Sobolev embedding $ W^{2, p} \hookrightarrow \calC^{1, \alpha} $. Fix $ \delta $ such that $ \delta C C' < \kappa $. Finally, we choose an associated $ F $ in the sense of Lemma \ref{Set:lemma2}.

The solvability of (\ref{BD:eqn2}) follows exactly the same argument in Proposition \ref{Set:prop1}. The regularity $ u \in \calC^{\infty}(W) $ also follows exactly. 

The standard $ \calL^{p} $-regularity theory given in \cite{Niren4} also holds for compact manifolds with smooth boundary, which follows that
\begin{equation*}
    \lVert u \rVert_{W^{2, p}(W, \sigma^{*} \bar{g})} \leqslant C \left( \lVert F \rVert_{\calL^{p}(W, \sigma^{*} \bar{g})} + \lVert u \rVert_{\calL^{p}(W, \sigma^{*}\bar{g})} \right).
\end{equation*}
Due to the injectivity of the operator, a very similar blow-up argument of \cite[Proposition 2.1]{RX2} shows that $ \lVert u \rVert_{\calL^{p}(W, \sigma^{*} \bar{g})} $ can be bounded above by $ \lVert Lu \rVert_{\calL^{p}(W, \sigma^{*}\bar{g})} $. Since $ R > 0 $, we observe that the zeroth order term $ Ru $ dominates the Neumann boundary condition so that our operator $ \left(L, \frac{\partial}{\partial \nu} \right) $ is injective on the compact space $ W $. It follows that the $ \calL^{p}$ estimates of our solution of (\ref{BD:eqn2}) can also be improved by
\begin{equation*}
     \lVert u \rVert_{W^{2, p}(W, \sigma^{*} \bar{g})} \leqslant C \lVert F \rVert_{\calL^{p}(W, \sigma^{*} \bar{g})}.
\end{equation*}
The rest of the argument in Proposition \ref{Set:prop1} then follows exactly. The only difference is that we need to apply the associated Sobolev embedding inequalities on compact manifolds with boundary given by Aubin \cite[Chapter 2]{Aubin}.
\end{proof}
Denote by $ (N, \partial N, \hat{g}) $ a $ n $-dimensional noncompact manifold with $ (n - 1) $-dimensional nonempty, smooth, noncompact boundary $ \partial N $ and a complete metric $ \hat{g} $. Let $ H^{1}(N, \hat{g}) $ be the standard Sobolev space on $ (N, \partial N, \hat{g}) $. Let $ \calC^{\infty}(N) $ be the space of smooth functions on $ N $, every element of $ \calC^{\infty}(N) $ is smooth up to the boundary. The mean curvature $ h_{e^{2f}\hat{g}} $ for the Yamabe metric $ e^{2f}\hat{g}, f \in \calC^{\infty}(N) $ satisfies
\begin{equation*}
h_{e^{2f}\hat{g}} = e^{-f} \left( h_{\hat{g}} + \frac{\partial f}{\partial \nu_{\hat{g}}}  \right).
\end{equation*}
Here $ \nu_{\hat{g}} $ is the unit normal vector field along $ \partial N $. Setting $ u^{\frac{4}{n - 2}} = e^{2f} $, it suggests that, as mentioned in \cite{ESC}, our notion of positivity on $ (N, \partial N, \hat{g}) $ must associate with the conformal Laplacian with an appropriate Robin boundary condition
\begin{align*}
L_{\text{conf}, \hat{g}} v & : = -\frac{4(n -1)}{n - 2} \Delta_{\hat{g}} v + R_{\hat{g}} v \; {\rm in} \; N, \\
B_{\text{conf}, \hat{g}} v & : = \frac{\partial v}{\partial \nu_{\hat{g}}} + \frac{n - 2}{2} h_{\hat{g}} v \; {\rm on} \; \partial N.
\end{align*} 
We introduce a notion of positivity for $ (N, \partial N, \hat{g}) $. A version of Yamabe constant on $ (N, \partial N, \hat{g}) $ that is compatible with the pair of the operators $ (L_{\text{conf}, \bar{g}}, B_{\text{conf}, \hat{g}} ) $ is given by:
\begin{equation}\label{BD:eqn4}
\begin{split}
& Y(N, \partial N, \hat{g}) \\
& \qquad : = \inf_{v \in (H^{1}(N, \hat{g}) \cap \calC^{\infty}(N)) \backslash \lbrace 0 \rbrace} \frac{\frac{4(n - 1)}{n - 2}\int_{N} \lvert \nabla_{\hat{g}} v \rvert^{2} d\text{Vol}_{\hat{g}} + \int_{N} R_{\hat{g}} v^{2} d\text{Vol}_{\hat{g}} + 2(n - 1) \int_{\partial N} h_{g} v^{2} dS_{\hat{g}}}{\lVert v \rVert_{\calL^{\frac{2n}{n - 2}}(N, \hat{g})}^{2}} \\
& \qquad : = \inf_{v \in (H^{1}(N, \hat{g}) \cap \calC^{\infty}(N)) \backslash \lbrace 0 \rbrace} \hat{E}_{\hat{g}}(v).
\end{split}
\end{equation}
Here $ dS_{\hat{g}} $ is the volume form on $ \partial M $ in terms of the induced metric $ \imath^{*} \hat{g} $. Clearly $ Y(N, \partial N, \hat{g}) $ is invariant under conformal transformation, including scaling of metrics. 

With the notions of positivity, we construct the conformal factor that will be used on $ (M \times \mathbb{S}_{t}^{1}, \partial M \times \mathbb{S}_{t}^{1}, \bar{g}) $. We use the same notations as in \S2. We pullback the same cut-off function $ \phi $ defined in (\ref{Set:eqnC0}) to a smooth function $ \Phi : M \times \mathbb{S}_{t}^{1} \rightarrow \R $, up to the boundary, by the projection $ \Pi_{1}: M \times \mathbb{S}_{t}^{1} \rightarrow \R $. Again, we set
\begin{equation}\label{BD:eqn5}
\begin{split}
    u_{0} & = 1 + u, u_{1} = \pi^{*} u_{0}, u' : =  u_{1}, \tilde{u}' = \tau_{1}^{*} u' \\
    \sigma^{*}u' & = u' |_{W} = u_{1} |_{W} = u_{0}, \imath^{*}\tilde{u}' = \imath^{*} \tau_{1}^{*} u' = \tau_{2}^{*} \sigma^{*} u' = u_{0} |_{X \times \lbrace 0 \rbrace_{\xi} \times \lbrace P \rbrace_{t}} \\
    \bar{u}_{1} & = \pi^{*} u, \bar{u}' = \Phi \cdot u_{1}, \bar{u}' |_{W \times [-1, 1]_{\xi}} = \bar{u}_{1}, \\
    \sigma^{*}\bar{u}' & = \bar{u}' |_{W} = (\Phi \cdot u_{1}) |_{W} = u, \imath^{*} \tau_{1}^{*} \bar{u}' = \tau_{2}^{*} \sigma^{*} \bar{u}' = u |_{X \times \lbrace 0 \rbrace_{\xi} \times \lbrace P \rbrace_{t}}.
\end{split}
\end{equation}
The key step for our codimension two approach is the partial $ \calC^{2} $-estimate for the function $ \frac{\partial^{2}u}{\partial t^{2}} $ on $ X $, up to the boundary. Since $ (X, \partial X) $ is a compact set, we can obtain our boundedness by local argument. We divide our argument into two cases: one for interior charts of $ (X, \partial X) $, and the other for boundary charts. For local elliptic estimates, we apply a fabulous result of Agmon, Douglis and Nirenberg \cite[Theorem 7.3]{Niren4}, where they generalized the interior estimates to involve information on a portion of the entire boundary. Analogous to Lemma \ref{Set:lemma4}, we need the following result:
\begin{lemma}\label{BD:lemma2}
Choosing $ \epsilon $ defined in Lemma \ref{Set:lemma2} to be small enough that will be determined below. Assume that $ (M, \partial M, g) $ is a noncompact cylinder with boundary, where $ g $ is of bounded geometry such that $ Y(M, \partial M, g) > 0 $. Let $ (W, \partial W, \sigma^{*} \bar{g}) $ be the associated space defined in (\ref{diagram1}). Let $ p, C, \delta $ and $ F $ be the same as in Proposition \ref{BD:prop1}. If $ u $ is the associated solution of (\ref{BD:eqn2}), then there exists a constant $ \eta' \ll 1 $ such that
\begin{equation}\label{BD:eqn6}
    \left\lVert \frac{\partial^{2} u}{\partial t^{2}} \right\rVert_{\calC^{0}\left(X \times \lbrace 0 \rbrace_{\xi} \times \left(-\frac{\epsilon}{4}, \frac{\epsilon}{4} \right)_{t} \right)} < \eta',
\end{equation}
up to the boundary.
\end{lemma}
\begin{proof}
Let $s, p, n, \alpha, F $ be the same as in Lemma \ref{Set:lemma4}. Due to the nonempty boundary $ \partial W $, we have to improve the local elliptic estimates; we also have to improve the global estimate for $ \lVert u \rVert_{\calL^{\frac{2(n + 1)}{n - 1}(W, \sigma^{*}(\epsilon^{-2} \bar{g}))}} $.

Set $ \Lambda = X \times \lbrace 0 \rbrace_{\xi} \times \left( -\frac{\epsilon}{4}, \frac{\epsilon}{4} \right) $, $ \Pi = X \times \lbrace 0 \rbrace_{\xi} \times \left( -\frac{\epsilon}{2}, \frac{\epsilon}{2} \right) $, $ \Omega = X \times \lbrace 0 \rbrace_{\xi} \times \left( -\frac{2\epsilon}{3}, \frac{2\epsilon}{3} \right) $, and $ \Lambda' = X \times \lbrace 0 \rbrace_{\xi} \times \left( -\frac{1}{4}, \frac{1}{4} \right) $, $ \Pi' = X \times \lbrace 0 \rbrace_{\xi} \times \left( -\frac{1}{2}, \frac{1}{2} \right) $, $ \Omega' = X \times \lbrace 0 \rbrace_{\xi} \times \left( -\frac{2}{3}, \frac{2}{3} \right) $. Again by Lemma \ref{Set:lemma2}, $ F \equiv C + 1 $ on $ \Pi $. We apply $ \frac{\partial^{2}}{\partial t^{2}} $ on both sides of (\ref{Set:eqn3}) in $ \Pi $, 
\begin{equation*}
  L_{\sigma^{*} \bar{g}} \left( \frac{\partial^{2}u}{\partial t^{2}} \right) : =  4\nabla_{V} \nabla_{V} \left( \frac{\partial^{2}u}{\partial t^{2}} \right) - 4\Delta_{\sigma^{*} \bar{g}}  \left( \frac{\partial^{2}u}{\partial t^{2}} \right) + R \left( \frac{\partial^{2}u}{\partial t^{2}} \right) = 0  \; {\rm in} \; \Pi, \frac{\partial \left( \frac{\partial^{2}u}{\partial t^{2}} \right)}{\partial \nu} = 0 \; {\rm on} \; \partial W \cap \partial \Pi.
\end{equation*}
Set $ t = \epsilon t' $, the new $ t' $-variable is associated with the metric $ (dt')^{2} = \epsilon^{-2} dt^{2} $ on $ \mathbb{S}^{1} $, we conclude by the same argument in Lemma \ref{Set:lemma4} that the following PDE holds in $ \Pi' $ with the new metric $ \bar{g} \mapsto \epsilon^{-2} \bar{g} $,
\begin{equation}\label{BD:eqn7}
\begin{split}
L_{\sigma^{*}(\epsilon^{-2}\bar{g})}\left( \frac{\partial^{2}u}{\partial (t')^{2}} \right) & : = 4\nabla_{V_{\epsilon^{-2}\bar{g}}} \nabla_{V_{\epsilon^{-2}\bar{g}}} \left( \frac{\partial^{2}u}{\partial (t')^{2}} \right) - 4\Delta_{\sigma^{*} (\epsilon^{-2}\bar{g})}  \left( \frac{\partial^{2}u}{\partial (t')^{2}} \right) + \epsilon^{2} R \left( \frac{\partial^{2}u}{\partial (t')^{2}} \right) = 0; \\
B_{\sigma^{*} (\epsilon^{-2} \bar{g})} \left( \frac{\partial ^{2} u}{\partial (t')^{2}} \right) & : = \frac{\partial \left( \frac{\partial ^{2} u}{\partial (t')^{2}} \right)}{\partial \nu_{\sigma^{*} (\epsilon^{-2}\bar{g})}} = 0 \; {\rm on} \; \partial W \cap \partial \Pi'.
\end{split}
\end{equation}
To estimate the $ \calC^{0}$-norm of $ \frac{\partial ^{2} u}{\partial (t')^{2}} $, we analyze the boundary charts first. Let $ U_{1}, U_{2}, U_{3} $ be boundary charts of $ W $ with respect to the original metric $ \sigma^{*} \bar{g} $ such that $ U_{1} \subset U_{2} \subset U_{3} $ and $ U_{1} \subset \Lambda $, $ U_{2} \subset \Pi $ and $ U_{3} \subset \Omega $. Accordingly, let $ U_{1}' \in \Lambda', U_{2}' \in \Pi', U_{3}' \in \Omega' $ be the scalinga of $ U_{1}, U_{2}, U_{3} $ with respect to $ \sigma^{*}(\epsilon^{-2} \bar{g}) $. Define $ \Gamma : = \partial U_{2}' \cap \partial W = \partial U_{2}' \cap \partial \Pi' \neq \emptyset $. By our setting, we observe that $ \partial U_{1}' \cap \partial U_{2}' \subset \Gamma $.
Therefore, with the same scaling argument as in Lemma \ref{Set:lemma4} and the elliptic estimates in \cite[Theorem 7.3]{Niren4}, we modify (\ref{Set:eqn10}) to get
\begin{equation}\label{BD:eqn8}
\begin{split}
    \left\lVert \frac{\partial ^{2} u}{\partial (t')^{2}} \right\rVert_{H^{s}(U_{1}', \sigma^{*}(\epsilon^{-2} \bar{g}))} & \leqslant D_{s, \epsilon} \left\lVert L_{\sigma^{*}(\epsilon^{-2}\bar{g})} \frac{\partial ^{2} u}{\partial (t')^{2}} \right\rVert_{H^{s - 2}(U_{2}', \sigma^{*}(\epsilon^{-2} \bar{g}))} \\
    & \qquad + D_{s, \epsilon}' \left\lVert B_{\sigma^{*}(\epsilon^{-2}\bar{g})} \frac{\partial ^{2} u}{\partial (t')^{2}} \right\rVert_{H^{s - 1}(\Gamma, \tau_{2}^{*} \sigma^{*}(\epsilon^{-2} \bar{g}))} + D_{s} \left\lVert \frac{\partial ^{2} u}{\partial (t')^{2}} \right\rVert_{\calL^{2}(U_{2}', \sigma^{*}(\epsilon^{-2} \bar{g}))} \\
    & = D_{s} \left\lVert \frac{\partial ^{2} u}{\partial (t')^{2}} \right\rVert_{\calL^{2}(U_{2}', \sigma^{*}(\epsilon^{-2} \bar{g}))}
\end{split}
\end{equation}
by (\ref{BD:eqn7}).
Denote interior charts by $ V_{1}, V_{2}, V_{3} $ of $ W $ such that $ V_{1} \subset V_{2} \subset V_{3} $, $ V_{1} \subset \Lambda $, $ V_{2} \subset \Pi $ and $ V_{3} \subset \Omega $. We also denote by $ V_{1}' \subset \Lambda', V_{2}' \subset \Pi', V_{3} \subset \Omega' $ with respect to $ \sigma^{*}(\epsilon^{-2}\bar{g}) $. For this case, $ \partial V_{1}' \cap \partial V_{2}' = \emptyset $, and $ \partial V_{2}' \cap \partial W = \emptyset $, therefore the same elliptic estimate implies that
\begin{equation}\label{BD:eqn9}
\left\lVert \frac{\partial^{2} u}{\partial (t')^{2}} \right\rVert_{H^{s}(V_{1}', \sigma^{*}(\epsilon^{-2} \bar{g}))} \leqslant D_{s} \left\lVert \frac{\partial^{2} u}{\partial (t')^{2}} \right\rVert_{\calL^{2}(V_{2}', \sigma^{*}(\epsilon^{-2} \bar{g}))}
\end{equation}
again by the PDE (\ref{BD:eqn7}). 

To estimate $ \lVert u \rVert_{\calL^{\frac{2(n + 1)}{n - 1}(W, \sigma^{*}(\epsilon^{-2} \bar{g}))}} $, we need
\begin{equation*}
    \lambda: = Y(M \times \mathbb{S}_{t}^{1}, \partial M \times \mathbb{S}_{t}^{1}, \bar{g}) > 0.
\end{equation*}
It follows easily from the fact that $ Y(M, \partial M, g) > 0 $ and $ \bar{g} = g \oplus dt^{2} $. By bounded geometry hypotheses, we have
\begin{align*}
   D_{4}^{-2} \sqrt{\det{(\epsilon^{-2}\bar{g})}} & \leqslant \sqrt{\det(\sigma_{\xi}^{*}(\epsilon^{-2} \bar{g}) \oplus d(\xi')^{2})} \leqslant D_{3}^{\frac{2(n + 1)}{n - 1}} \sqrt{\det{(\epsilon^{-2}\bar{g})}}, \forall \xi; \\
   D_{4}^{-2} dS_{\epsilon^{-2}\bar{g}} & \leqslant dS_{\sigma_{\xi}^{*}(\epsilon^{-2} \bar{g}) \oplus d(\xi')^{2}} \leqslant D_{3}^{\frac{2(n + 1)}{n - 1}} dS_{\epsilon^{-2}\bar{g}}, \forall \xi.
\end{align*}
With the positivity of $ \lambda $ for noncompact manifolds with boundary, we improve (\ref{Set:A1}) by
\begin{equation}\label{BD:A1}
\begin{split}
& \lVert u \rVert_{\calL^{\frac{2(n + 1)}{n - 1}}(W, \sigma^{*} (\epsilon^{-2} \bar{g}))}^{2} \\
& \qquad = (2\epsilon)^{\frac{n -1}{n + 1}} \left( \int_{-\epsilon^{-1}}^{\epsilon^{-1}} \int_{W} \lvert \bar{u}_{1} \rvert^{\frac{2(n + 1)}{n - 1}} d\text{Vol}_{\sigma^{*} (\epsilon^{-2} \bar{g})} d\xi'  \right)^{\frac{n - 1}{n + 1}} \\
& \qquad = (2\epsilon)^{\frac{n -1}{n + 1}} \left( \int_{W \times [-\epsilon^{-1}, \epsilon^{-1}]} \lvert \bar{u}_{1} \rvert^{\frac{2(n + 1)}{n - 1}} d\text{Vol}_{\sigma^{*} \bar{g} \oplus d(\xi')^{2}}  \right)^{\frac{n - 1}{n + 1}} \\
& \qquad = (2\epsilon)^{\frac{n -1}{n + 1}} \left( \int_{W \times [-\epsilon^{-1}, \epsilon^{-1}]} \lvert \Phi(\cdot, \xi') \cdot \bar{u}_{1} \rvert^{\frac{2(n + 1)}{n - 1}} d\text{Vol}_{\sigma^{*} \bar{g} \oplus d(\xi')^{2}}  \right)^{\frac{n - 1}{n + 1}}  \\
& \qquad \leqslant (2\epsilon)^{\frac{n -1}{n + 1}} D_{3}^{2}  \left( \int_{M \times \mathbb{S}^{1}_{t'}}  \lvert \bar{u}' \rvert^{\frac{2(n + 1)}{n - 1}} d\text{Vol}_{\epsilon^{-2}\bar{g}} \right)^{\frac{n - 1}{n + 1}} \\
& \qquad \leqslant (2\epsilon)^{\frac{n -1}{n + 1} } D_{3}^{2} \lambda^{-1} 
 \left( \frac{4n}{n - 1} \lVert \nabla_{\epsilon^{-2}\bar{g}} \bar{u}' \rVert_{\calL^{2}(M \times \mathbb{S}^{1}_{t'}, \epsilon^{-2}\bar{g})} + \int_{M \times \mathbb{S}^{1}_{t'}} R_{\epsilon^{-2}\bar{g}} (\bar{u}')^{2} d\text{Vol}_{\epsilon^{-2}\bar{g}} \right) \\
 & \qquad \qquad + (2\epsilon)^{\frac{n -1}{n + 1} } D_{3}^{2} \lambda^{-1} \cdot 2(n - 1) \int_{\partial M \times \mathbb{S}_{t'}^{1}} h_{\epsilon^{-2} \bar{g}} (\bar{u}')^{2} dS_{\epsilon^{-2}\bar{g}} \\
& \qquad \leqslant (2\epsilon)^{\frac{n -1}{n + 1}} D_{3}^{2} D_{4}^{2} \lambda^{-1} \cdot \frac{4n}{n - 1}\int_{\R} \int_{W_{\xi}} \lvert \nabla_{\sigma_{\xi}^{*}(\epsilon^{-2}\bar{g})} \bar{u} \rvert^{2} d\text{Vol}_{\sigma_{\xi}^{*}(\epsilon^{-2}\bar{g})} d\xi'  \\
& \qquad \qquad + (2\epsilon)^{\frac{n -1}{n + 1}} D_{3}^{2} D_{4}^{2} \lambda^{-1} \int_{\R} \int_{W_{\xi}} R_{\epsilon^{-2}\bar{g}} \bar{u}^{2} d\text{Vol}_{\sigma^{*} (\epsilon^{-2}\bar{g})} d\xi'  \\
& \qquad \qquad \qquad + (2\epsilon)^{\frac{n -1}{n + 1}} D_{3}^{2} D_{4}^{2} \lambda^{-1} \cdot 2(n - 1) \int_{\R} \int_{\partial W_{\xi}}  h_{\epsilon^{-2} \bar{g}} (\bar{u}')^{2} dS_{\sigma^{*}(\epsilon^{-2} \bar{g})} d\xi' \\
& \qquad \leqslant 2^{\frac{n - 1}{n + 1}} D_{3}^{2} D_{4}^{2} D_{\Phi} \lambda^{-1} \epsilon^{-\frac{2}{n + 1}} \cdot \frac{4n}{n - 1} \max_{\xi \in I_{\xi}} \left( \lVert \nabla_{\sigma_{\xi}^{*} (\epsilon^{-2} \bar{g})}\bar{u}_{1} \rVert_{\calL^{2}(W_{\xi}, \sigma_{\xi}^{*}(\epsilon^{-2}\bar{g}))}^{2} \right) \\
& \qquad \qquad  + 2^{\frac{n - 1}{n + 1}} D_{3}^{2} D_{4}^{2} D_{\Phi} \lambda^{-1} \epsilon^{-\frac{2}{n + 1}} \cdot \epsilon^{2} \max_{W} \lvert R_{\bar{g}} \rvert \lVert u \rVert_{\calL^{2}(W, \sigma^{*}(\epsilon^{-2}\bar{g}))}^{2} \\
& \qquad \qquad \qquad  + 2^{\frac{n - 1}{n + 1}} D_{3}^{2} D_{4}^{2} D_{\Phi} \lambda^{-1} \epsilon^{-\frac{2}{n + 1}} \cdot 2(n - 1) \epsilon \max_{\partial W} \lvert h_{\bar{g}} \rvert \lVert u \rVert_{\calL^{2}(\partial W, \imath^{*}\sigma^{*}(\epsilon^{-2} \bar{g}))}^{2}.
\end{split}
\end{equation}
With (\ref{BD:A1}) and all other inequalities that can be applied directly for noncompact manifold with boundary case, we improve (\ref{Set:A2}) by
\begin{equation}\label{BD:A2}
\begin{split}
\left\lVert \frac{\partial^{2}u}{\partial (t')^{2}} \right\rVert_{\calC^{0}(O_{1})} & \leqslant D_{0}\epsilon^{\frac{n}{2}} \left\lVert \frac{\partial^{2}u}{\partial (t')^{2}} \right\rVert_{H^{s}(O_{1}, \sigma^{*}(\epsilon^{-2} \bar{g}))} \leqslant D_{0} D_{s} \epsilon^{\frac{n}{2}} \left\lVert \frac{\partial^{2}u}{\partial (t')^{2}} \right\rVert_{\calL^{2}(O_{2}, \sigma^{*}(\epsilon^{-2} \bar{g}))} \\
& \leqslant \epsilon^{2 + \frac{n}{2}} D_{0}D_{2}D_{s}  \lVert F \rVert_{\calL^{2}(W, \sigma^{*}(\epsilon^{-2}g))} + \epsilon^{2 - \frac{n}{n + 1} + \frac{n}{2}} D_{0}D_{1}D_{2}D_{s} \lVert u \rVert_{\calL^{\frac{2(n + 1)}{n - 1}}(W, \sigma^{*} (\epsilon^{-2} \bar{g}))} \\
& \leqslant 2^{\frac{n - 1}{2(n + 1)}} D_{0} D_{1} D_{2} D_{3}D_{4} D_{s} D_{\Phi} \lambda^{-\frac{1}{2}}  \epsilon^{1 + \frac{n}{2}} \times \\
& \qquad \times \left( \frac{4n}{n - 1} \max_{\xi \in I_{\xi}} \left( \lVert \nabla_{\sigma_{\xi}^{*} (\epsilon^{-2} \bar{g})} \bar{u} \rVert_{\calL^{2}(W_{\xi}, \sigma_{\xi}^{*}(\epsilon^{-2}\bar{g}))}^{2} \right) + \epsilon^{2} \max_{W} \lvert R_{\bar{g}} \rvert \lVert u \rVert_{\calL^{2}(W, \sigma^{*}(\epsilon^{-2}\bar{g}))}^{2} \right)^{\frac{1}{2}} \\
& \qquad \qquad + 2^{\frac{n - 1}{2(n + 1)}} D_{0} D_{1} D_{2} D_{3}D_{4} D_{s} D_{\Phi} \lambda^{-\frac{1}{2}}  \epsilon^{1 + \frac{n}{2}} \times \\
& \qquad \qquad \qquad \times \left( 2(n - 1)\epsilon \max_{\partial W} \lvert h_{\bar{g}} \rvert \lVert u \rVert_{\calL^{2}(\partial W, \imath^{*} \sigma^{*} \bar{g})} \right)^{\frac{1}{2}} \\
& \qquad \qquad \qquad \qquad + \epsilon^{2 + \frac{n}{2}} D_{0}D_{2}D_{s}  \lVert F \rVert_{\calL^{2}(W, \sigma^{*}(\epsilon^{-2}\bar{g}))} \\
& : = \bar{D}_{0}  \epsilon^{2} \left( \frac{4n}{n - 1} \max_{\xi \in I_{\xi}} \left( \lVert \nabla_{\sigma_{\xi}^{*}  \bar{g}} \bar{u} \rVert_{\calL^{2}(W_{\xi}, \sigma_{\xi}^{*}\bar{g})}^{2} \right) + \max_{W} \lvert R_{\bar{g}} \rvert \lVert u \rVert_{\calL^{2}(W, \sigma^{*}\bar{g})}^{2} \right)^{\frac{1}{2}} \\
& \qquad + \bar{D}_{0}  \epsilon^{2} \left( \max_{\partial W} \lvert h_{\bar{g}} \rvert \lVert u \rVert_{\partial W, \imath^{*} \sigma^{*} \bar{g}} \right)^{\frac{1}{2}} + \epsilon^{2 + \frac{n}{2}} D_{0}D_{2}D_{s}  \lVert F \rVert_{\calL^{2}(W, \sigma^{*}(\epsilon^{-2}\bar{g}))} \\
& < \frac{8n}{n - 1} \bar{D}_{0} D_{5} \epsilon^{2} \eta + D_{0}D_{2}D_{s} \epsilon^{2} \text{Vol}_{g}(X) (C + 1) \epsilon^{\frac{1}{2}} \\
& : = \bar{D} \epsilon^{2} \eta + \bar{D}' \epsilon^{2 + \frac{1}{2}} (C + 1).
\end{split}
\end{equation}
\end{proof}
The comparisons among Laplacians in Lemma \ref{Set:lemma5} hold without any change. We are now ready to prove the main theorem of this section.
\begin{theorem}\label{BD:thm2}
    Let $ (X, \partial X) $ be an oriented, compact manifold with non-empty smooth boundary. Let $ (M, \partial M, g) = (X \times \R, \partial X \times \R, g) $ be a Riemannian noncompact cylinder of bounded geometry such that $ n : = \dim M \geqslant 3 $. If $ (M, \partial M, g) $ has positive Yamabe constant in the sense of (\ref{BD:eqn4}), and
    \begin{equation*}
        \angle_{g}(\nu_{g}, \partial_{\xi}) \in [0, \frac{\pi}{4}),
    \end{equation*}
    along the hypersurface $ (X \times \lbrace 0 \rbrace, \partial X \times \lbrace 0 \rbrace ) : = X_{0} $, then there exists a metric $ \tilde{g} $ in the conformal class of $ g $ such that $ \imath^{*} \tilde{g} $ is a PSC metric with nonnegative mean curvature on $ (X, \partial X) \cong X_{0} $. Equivalently, the Yamabe constant of the conformal class $ [\imath^{*}g] $ on $ (X, \partial X) $ is positive.
\end{theorem}
\begin{proof}
As in the proof of Theorem \ref{PSC:thm1}, we consider the conformal transformation
\begin{equation*}
    g \mapsto \left( \tilde{u}' \right)^{\frac{4}{n -2}} g : = \tilde{g} \Rightarrow
    \imath^{*}g \mapsto u_{0}^{\frac{4}{n - 2}} \imath^{*}g : = \imath^{*} \tilde{g} 
\end{equation*}
on $ (X, \partial X) $. Let's check the mean curvature of the new metric $ \imath^{*} \tilde{g} $ first. By (\ref{BD:eqn2}),
\begin{equation*}
    \frac{\partial u_{0}}{\partial \nu} = \frac{\partial (u + 1)}{\partial \nu} = 0 \; {\rm on} \; \partial X.
\end{equation*}
Since $ h_{\imath^{*}g} = 0 $, it follows that
\begin{equation*}
    h_{\imath^{*} \tilde{g}} = u_{0}^{-\frac{2}{n - 2}} h_{\imath^{*}g} - \frac{\partial \left( u_{0}^{-\frac{2}{n - 2}}\right)}{\partial \nu} = 0 \; {\rm on} \; \partial X.
\end{equation*}
For the scalar curvature $ R_{\imath^{*}\tilde{g}} $, we follow exactly the same procedures in the proof of Theorem \ref{PSC:thm1}. We set $ p, 
\alpha $ as in Proposition \ref{BD:prop1}. We adjust the choice $ C $ by
\begin{equation}\label{BD:eqn10}
C > 2 \max_{X \times \lbrace 0 \rbrace_{\xi} \times \lbrace P \rbrace_{t}} \left( \lvert R_{g} \rvert + R +  2 \lvert {\rm Ric}_{g}(\nu_{g}, \nu_{g}) \rvert + h_{g}^{2} + \lvert A_{g} \rvert^{2} \right)  + 3.
\end{equation}
We then construct $ F $ such that (\ref{POS:eqn1a}) holds. All following arguments are the same, up to (\ref{POS:eqn3}). The differential relation that $ u_{0} $ satisfies is:
\begin{equation*}
    4\nabla_{V} \nabla_{V} u_{0} - 4\Delta_{\sigma^{*} \bar{g}} u_{0} + Ru_{0} = F + R \; {\rm on} \; X.
\end{equation*}
We revise (\ref{POS:eqn3}) by
\begin{align*}
    R_{\imath^{*} \tilde{g}} & \geqslant u_{0}^{-\frac{n +2}{n - 2}}\left( - 2{\rm Ric}_{\tau_{1}^{*}\bar{g}}\left(\nu_{g}, \nu_{g} \right)u_{0} + h_{\tau_{1}^{*} \bar{g}}^{2} u_{0} - \lvert A_{\tau_{1}^{*} \bar{g}} \rvert^{2} u_{0} \right) \\
    & \qquad + u_{0}^{-\frac{n +2}{n - 2}}\left( 4\nabla_{V} \nabla_{V} u_{0} - 4\Delta_{\sigma^{*}\bar{g}} u_{0}  + R_{\bar{g}} |_{W} u_{0}\right) + u_{0}^{-\frac{n +2}{n - 2}} \cdot (-1 - 4 \eta' + A_{1}(\bar{g}, M, u)) \\
     & = u_{0}^{-\frac{n +2}{n - 2}}\left( - 2{\rm Ric}_{\tau_{1}^{*}\bar{g}}\left(\nu_{g}, \nu_{g} \right)u_{0} + h_{\tau_{1}^{*} \bar{g}}^{2} u_{0} - \lvert A_{\tau_{1}^{*} \bar{g}} \rvert^{2} u_{0} \right) \\
    & \qquad + u_{0}^{-\frac{n +2}{n - 2}}\left( 4\nabla_{V} \nabla_{V} u_{0} - 4\Delta_{\sigma^{*}\bar{g}} u_{0}  + R u_{0}\right) \\
    & \qquad \qquad + u_{0}^{-\frac{n +2}{n - 2}} \cdot ( - Ru_{0} + R_{\bar{g}} |_{W} u_{0} -1 - 4 \eta' + A_{1}(\bar{g}, M, u)) \\
    & =  u_{0}^{-\frac{n +2}{n - 2}}\left(- 2{\rm Ric}_{\tau_{1}^{*}\bar{g}}\left(\nu_{g}, \nu_{g} \right)u_{0} + h_{\tau_{1}^{*} \bar{g}}^{2} u_{0} - \lvert A_{\tau_{1}^{*} \bar{g}} \rvert^{2} u_{0} - Ru_{0} + R_{\bar{g}} |_{W} u_{0} \right) \\
    & \qquad u_{0}^{-\frac{n +2}{n - 2}} \left(F + R -1 - 4 \eta' + A_{1}(\bar{g}, M, u) \right).
\end{align*}
It follows that
\begin{equation*}
    R_{\imath^{*}\tilde{g}} \geqslant u_{0}^{\frac{n + 2}{n - 2}} \left( C + 1 + R -  2 \max_{X \times \lbrace 0 \rbrace_{\xi} \times \lbrace P \rbrace_{t}} \left( \lvert R_{g} \rvert + R +  2 \lvert {\rm Ric}_{g}(\nu_{g}, \nu_{g}) \rvert + h_{g}^{2} + \lvert A_{g} \rvert^{2} \right)  - 3 \right) > 0
\end{equation*}
on $ X $. Due to Escobar \cite{ESC}, $ R_{\imath^{*} \tilde{g}} > 0 $ and $ h_{\imath^{*}\tilde{g}} = 0 $ implies that the Yamabe constant with respect to the conformal class $ [\imath^{*}\tilde{g}] $ is positive. Clearly $[\imath^{*}\tilde{g}] = [\imath^{*}g] $ since we take the conformal transformation.
\end{proof}

It is straightforward that the same argument of Theorem \ref{BD:thm2} holds when $ \partial Z = \emptyset $ by removing the boundary condition of (\ref{BD:eqn1}) in the argument of Proposition \ref{BD:prop1}. Hence we can drop the uniform PSC assumption in Corollary \ref{BD:cor1}:
\begin{corollary}\label{BD:cor1}
Let $ X $ be an oriented, closed manifold. Let $ (M = X \times \R, g) $ be a Riemannian noncompact cylinder of bounded geometry or bounded curvature such that $ n : = \dim M \geqslant 3 $. If the Yamabe constant of the conformal class $ [g] $ on $ M $ is positive such that
    \begin{equation*}
        \angle_{g}(\nu_{g}, \partial_{\xi}) \in [0, \frac{\pi}{4})
    \end{equation*}
along the hypersurface $ X_{0} $, then the Yamabe constant of the conformal class $ [\imath^{*}g] $ on $ X $ is also positive.
\end{corollary}
  
\bibliographystyle{plain}
\bibliography{ScalarPre}

\begin{thebibliography}{10}

\bibitem{Niren4}
S.~Agmon, A.~Douglis, and L.~Nirenberg.
\newblock Estimates near the boundary for solutions of elliptic partial
  differential equstions satisfying general boundary conditions {I}.
\newblock {\em Commun. Pure Appl. Math}, 12:623--727, 1959.

\bibitem{Aubin}
T.~Aubin.
\newblock {\em Nonlinear Analysis on Manifolds. {M}onge-{A}mp\'ere
  {E}quations.}
\newblock Grundlehren der mathematischen Wissenschaften. Springer, Berlin,
  Heidelberg, New York, 1982.

\bibitem{CRZ}
S.~Cecchini, D.~R\"ade, and R.~Zeidler.
\newblock Nonnegative scalar curvature on manifolds with at least two ends.
\newblock {\em J. Topol.}, 16:855--876, 2023.

\bibitem{CZ}
S.~Cecchini and R.~Zeidler.
\newblock Scalar and mean curvature comparison via the {D}irac operator.
\newblock {\em Geometry and Topology}, 28:1167--1212, 2024.

\bibitem{CL}
O.~Chodosh and C.~Li.
\newblock Generalized soap bubbles and the topology of manifolds with positive
  scalar curvature.
\newblock {\em Ann. of Math. (2)}, 199:707--740, 2024.

\bibitem{CLL}
O.~Chodosh, C.~Li, and Y.~Liokumovich.
\newblock Classifying sufficiently connected psc manifolds in 4 and 5
  dimensions.
\newblock {\em Geom. Topol.}, 27:1635–1655, 2023.

\bibitem{ESC}
J.~Escobar.
\newblock The {Y}amabe problem on manifolds with boundary.
\newblock {\em J. Differential Geom.}, 35:21--84, 1992.

\bibitem{GROMOV3}
M.~Gromov.
\newblock No metrics with positive scalar curvatures on aspherical 5-manifolds.
\newblock {\em ar{X}iv:2009.05332}.

\bibitem{GROMOV2}
M.~Gromov.
\newblock Metric inequalities with scalar curvature.
\newblock {\em Geom. Funct. Anal.}, 28(3):645--726, 2018.

\bibitem{GL}
M.~Gromov and B.~Lawson.
\newblock Spin and scalar curvature in the presence of a fundamental group.
  {I}.
\newblock {\em Ann. of Math. (2)}, 111(2):209--230, 1980.

\bibitem{Grosse}
N.~Grosse.
\newblock The {Y}amabe equation on manifolds of bounded geometry.
\newblock {\em Comm. Anal. Geom.}, 21(5):957--978, 2013.

\bibitem{HPS}
B.~Hanke, D.~Pape, and T.~Schick.
\newblock Codimension two index obstructions to positive scalar curvature.
\newblock {\em Ann. Inst. Fourier (Grenoble)}, 65:2681–2710, 2015.

\bibitem{Hebey}
E.~Hebey.
\newblock {\em Sobolev spaces on Riemannian manifolds}.
\newblock Lect. Notes Math., 1635. Springer-Verlag, Berlin, 1996.

\bibitem{Kim1}
S.~Kim.
\newblock Scalar curvature on noncompact complete {R}iemannian manifolds.
\newblock {\em Nonlinear Analysis, Theory, Methods and Applications},
  26(12):1985--1993, 1996.

\bibitem{PL}
J.~Lee and T.~Parker.
\newblock The {Y}amabe problem.
\newblock {\em Bull. Amer. Math. Soc. (N.S.)}, 17(1):37--91, 1987.

\bibitem{Rade}
D.~R\"ade.
\newblock Scalar and mean curvature comparison via $ \mu $-bubbles.
\newblock {\em Calc. Var. Partial Differential Equations}, 62(187), 2023.

\bibitem{JR}
J.~Rosenberg.
\newblock Manifolds of positive scalar curvature: a progress report.
\newblock {\em Surveys in Differential Geometry}, 11(1):259--294, 2006.

\bibitem{RosSto}
J.~Rosenberg and S.~Stolz.
\newblock Manifolds of positive scalar curvature.
\newblock {\em Algebraic topology and its applications. Vol. 27. Math. Sci.
  Res. Inst. Publ}, 27:241--267, 1994.

\bibitem{RX2}
S.~Rosenberg and J.~Xu.
\newblock A codimension two approach to the $ \mathbb{S}^{1} $-stability
  conjecture.
\newblock {\em ar{X}iv:2412.12479}.

\bibitem{Schick}
T.~Schick.
\newblock Manifolds with boundary and of bounded geometry.
\newblock {\em Math. Nachr.}, 223:103--120, 2001.

\bibitem{SY2}
R.~Schoen and S.T.Yau.
\newblock On the structure of manifolds with positive scalar curvature.
\newblock {\em Manuscripta Math.}, 28:159--183, 1979.

\bibitem{Wei}
G.~Wei.
\newblock Yamabe equation on some complete nomcompact manifolds.
\newblock {\em Pacific Journal of Mathematics}, 302(2):717--739, 2019.

\bibitem{XU10}
J.~Xu.
\newblock Scalar and mean curvature comparison on compact cylinder.
\newblock {\em ar{X}iv:2507.07005}.

\bibitem{Zeilder2}
R.~Zeidler.
\newblock An index obstruction to positive scalar curvature on fiber bundles
  over aspherical manifolds.
\newblock {\em Algebr. Geom. Topol.}, 17:3081–3094, 2017.

\bibitem{Zeidler}
R.~Zeidler.
\newblock Band width estimates via the {D}irac operator.
\newblock {\em J. Differential Geom.}, 122(1):155--183, 2022.

\end{thebibliography}
\end{document}